\documentclass[review, 11pt]{elsarticle}

\biboptions{authoryear}

\usepackage[english]{babel}
\usepackage{amsmath,amsthm,amsfonts,amssymb,latexsym,amstext}

\usepackage[dvipsnames]{xcolor} 
\usepackage{mathrsfs} 
\usepackage{natbib}
\usepackage{float} 
\usepackage{hyperref} 
\usepackage{url}
\usepackage{booktabs} 

\newtheoremstyle{thmstyle} 
{\topsep}                    
{\topsep}                    
{\itshape}                   
{}                           
{\bfseries}                   
{.}                          
{.5em}                       
{}  
\theoremstyle{thmstyle}
\newtheorem{theorem}{Theorem}[section]
\newtheorem{proposition}{Proposition}[section]
\newtheorem{lemma}{Lemma}[section]

\theoremstyle{definition}
\newtheorem{definition}{Definition}[section]
\newtheorem{Example}{Example}[section]
\newtheorem{remark}{Remark}[section]

\usepackage{multirow}

\newcommand{\addQEDstyle}[2]{\AtBeginEnvironment{#1}{\pushQED{\qed}\renewcommand{\qedsymbol}{#2}}\AtEndEnvironment{#1}{\popQED}}
\addQEDstyle{example}{$\triangle$}

\newcommand\restr[2]{{
  \left.\kern-\nulldelimiterspace 
  #1 
  \vphantom{\big|} 
  \right|_{#2} 
  }}
  
\usepackage{mathtools} 
\DeclarePairedDelimiter\abs{\lvert}{\rvert}

\newcommand{\restrict}[1]{|_{#1}} 

\newcommand\xqed[1]{
	\leavevmode\unskip\penalty9999 \hbox{}\nobreak\hfill
	\quad\hbox{#1}} 
\newcommand\sexample{\xqed{$\triangle$}} 

\usepackage[margin=1.5cm]{geometry}

 \usepackage{setspace} 
\makeatletter
\newcommand\thefontsize{The current font size is: \f@size pt}
\makeatother

\newpageafter{author}

\makeatletter
\def\ps@pprintTitle{%
	\let\@oddhead\@empty
	\let\@evenhead\@empty
	\def\@oddfoot{\reset@font\hfil\thepage\hfil}
	\let\@evenfoot\@oddfoot
}
\makeatother
\begin{document}

\begin{frontmatter}
\title{Cost allocation problems on highways with grouped users} 

\author[mymainaddress]{Marcos G{\'o}mez-Rodr{\'i}guez}
\ead{marcos.gomez.rodriguez@udc.es}

\author[mysecondaryaddress,mythirdaddress]{Laura Davila-Pena\corref{mycorrespondingauthor}}
\cortext[mycorrespondingauthor]{Corresponding author}
\ead{lauradavila.pena@usc.es and l.davila-pena@kent.ac.uk}

\author[myfourthaddress]{Balbina Casas-Méndez}
\ead{balbina.casas.mendez@usc.es}

\address[mymainaddress]{Department of Mathematics, Faculty of Computer Science, Universidade da Coruña, Campus de Elviña, 15071 A Coruña, Spain.}
\address[mysecondaryaddress]{MODESTYA Research Group, Department of Statistics, Mathematical Analysis and Optimization, Faculty of Mathematics, Universidade de Santiago de Compostela, Campus Vida, 15782 Santiago de Compostela, Spain.}
\address[mythirdaddress]{Department of Analytics, Operations and Systems, Kent Business School, University of Kent, CT2 7PE Canterbury, UK.}
\address[myfourthaddress]{CITMAga, MODESTYA Research Group, 
Department of Statistics, Mathematical Analysis and Optimization, Faculty of Mathematics, Universidade de Santiago de Compostela, Campus Vida, 15782 Santiago de Compostela, Spain.}

\begin{abstract} 

One of the practical applications of cooperative transferable utility games involves determining the fee structure for users of a given facility, whose construction or maintenance costs need to be recouped. In this context, certain efficiency and equity criteria guide the considered solutions. 
This paper analyzes how to allocate the fixed costs of a highway among its users through tolls, considering that different classes of vehicles or travelers utilize the service. For this purpose, we make use of generalized highway games with a priori unions that represent distinct user groups, such as frequent travelers or truckers, who, due to enhanced bargaining power, often secure reductions in their fares in real-world scenarios. In particular, the Owen value, the coalitional Tijs value, and a new value termed the Shapley-Tijs value, are axiomatically characterized. Additionally, straightforward formulations for calculating these values are provided. Finally, the proposed methodology is applied to actual traffic data from the AP-9 highway in Spain. 

\end{abstract}


\end{frontmatter}

\section{Introduction}

The current paper attempts to study a cost-sharing problem within the realm of transportation. The primary objective is to examine the distribution of highway construction costs among its users considering the existence of externalities. In particular, we use the so-called a priori unions to assess the bargaining power of different groups, including various  classes of vehicles (light, heavy 1, and heavy 2) or frequent travelers.

Cooperative games have diverse applications across several domains, with cost-sharing problems standing out prominently. Within the cost allocation problems, the work conducted by \cite{Fiestras2011} offers a comprehensive review of the literature concerning the applications of cooperative transferable utility games in this context. Within this expansive landscape, three sectors are specifically examined, constituting significant domains for such problems: the energy industry, natural resource management, and transportation. In the subsequent discussion, we highlight pertinent papers focused on the transport sector. 

\cite{Ozener2008} investigate a logistics network in which several shippers collaborate and pool their shipping requests to negotiate improved rates with a common carrier. Their research identifies optimal collaborative routes, minimizing overall empty truck movements, and introduces several cost-sharing mechanisms, encompassing properties familiar in cooperative game theory and novel contributions. One of the most popular works in the transportation field is \cite{Littlechild1973}, where airport games are defined, in which the costs associated with building a runway are allocated among the aircraft that use it. 
\cite{VazquezBrage1997} and \cite{CasasMendez2003} provide expressions for the Owen value \citep{Owen1977} and the coalitional Tijs value \citep{CasasMendez2003}, respectively, within the framework of airport games. These contributions propose distinct solutions to distributing runway costs, considering the bargaining power of each airline. Both values incorporate the model of a priori unions in cooperative games, as initially introduced by \cite{Owen1977}. 

The use of cooperative game theory in the pursuit of an optimal toll for highways has already been investigated in a number of papers, following different approaches. From the existing literature, the next studies are noteworthy. \cite{Villarreal1985} develop two highway cost allocation methodologies which extend the basic concepts of incremental and proportional allocation procedures and illustrate it by means of a real-life example. \cite{Makrigeorgis1991} investigates an equitable and stable rule for sharing the total cost of providing a highway service among users, employing game theory concepts such as the core and the marginal to rational max-min ratio nucleolus. Both of these works consider classes of vehicles as players. 
In \cite{Castano1995}, each vehicle passage is treated as an individual player, and the value of the corresponding non-atomic game is used to find the solution to the pavement cost allocation problem. The characteristic function of the game is formulated as a non-linear optimization problem. \cite{Dong2012} designate players as different input-output pairs of the highway, with each pair representing the vehicles entering through a given entrance and exiting through a given exit. The proposed pricing method charges each vehicle the total of its average costs across the highway sections it traverses. This approach adheres to important axioms and avoids prompting users to adopt strategic responses. 
In \cite{Mosquera2007}, \cite{Ciftcci2010y}, \cite{Kuipers2013}, and \cite{Sudholter2017}, players are considered to be the different trips taken. \cite{Ciftcci2010y} examine the concavity and balancedness of certain highway games on weakly cyclic graphs. \cite{Mosquera2007} and \cite{Kuipers2013} introduce the so-called highway games, a generalization of the airport games, and compute for them the Shapley value \citep{Shapley1953}, the Tijs value \citep{Tijs1981}, and the nucleolus \citep{Schmeidler1969}. In \cite{Kuipers2013}, the Shapley value and the nucleolus are applied to the apportionment of the fixed costs of the AP-68 highway (Spain) among its users. Nevertheless, its scope is restricted to the distribution of costs exclusively among light vehicles. By relaxing some assumptions of the model, 
\cite{Sudholter2017} formulate a new problem referred to as the generalized highway problem. They provide axiomatizations of the core, the prenucleolus, and the Shapley value on the class of games associated with these generalized highway problems. 
While the aforementioned studies focus on distributing highway costs among users, \cite{Brink2022} analyze how to allocate the total toll collected among the different highway sections. The authors propose and axiomatically characterize three allocation procedures, establishing a connection to the Shapley value of specific cooperative games with transferable utility, which, in this case, are benefit games rather than cost games.
 
The present work also considers generalized highway problems. Our first contribution is the introduction of a methodology designed to distribute highway costs among various types of users extending beyond the exclusive allocation to light vehicles. The approach adopted in our setup, to include the different classes of vehicles, is inspired by the methodology used in \cite{Fragnelli2000} to analyze a cost allocation problem, called the infrastructure cost problem. Such a problem emerged during the railway reorganization conducted in Europe in the 1990s. In \cite{Fragnelli2000}, the allocation of construction and maintenance costs for the railway network in Italy is investigated, involving two classes of trains with different requirements--fast and slow trains. Fast trains require high-quality tracks, while basic tracks suffice for slow trains. The construction costs are modeled as an airport game, where both train classes utilize the basic level of track (equivalent to the first section of the runway), but fast trains additionally need the next level of quality (corresponding to the full runway). In our study, we apply a similar conceptualization to light, heavy 1, and heavy 2 vehicles, since each vehicle type entails distinct road operational requirements. To illustrate, the expenses associated with constructing a bridge intended for use by small cars versus large trucks exhibit disparities. Generally, the design of a highway initially planned for a specific volume of light vehicles necessitates subsequent adaptations to accommodate varying volumes of heavier vehicles \citep{Villarreal1985}. Thus, in our setup, light, heavy 1, and heavy 2 vehicles will now traverse a highway consisting of three road levels, aligning with the two-level framework introduced in the railway infrastructure allocation model. 

The inclusion of different levels in our problem prompts the consideration of an airport game for each of them, making it belong to the class of generalized highway games. This classification is attributed to our model's assumption that users can utilize disconnected sections of the highway. Similar to the approach in \cite{Kuipers2013}, our proposed model excludes consideration of maintenance costs. However, our methodology establishes a hierarchical pricing structure, wherein the fee for heavy 2 vehicles exceeds that of heavy 1 vehicles, and likewise, the fee for heavy 1 vehicles surpasses that of light vehicles. This aligns with real-world scenarios, accounting for the varying weights of vehicle types and their consequential impact on infrastructure. The coherence of this approach becomes apparent through the outcomes derived in the application discussed in Section~\ref{sec:ap9}. It will be observed that the fares calculated using our methodology and applying the Shapley value coincide with the actual fares for each of the three vehicle types considered in each of the analyzed sections.

The second contribution of this paper is motivated by recent negotiations in the context of the AP-9 highway (Spain), leading to special fares for certain groups, such as truck drivers or users undertaking round trips or more than 20 trips per month. Consequently, in conjunction with the generalized highway game, we propose employing a priori unions to model the bargaining power of these user groups. We follow the approach of \cite{Owen1977}, who introduced a model accounting for potential affinities between players in a general cooperative game and how these affinities can influence the distribution of costs.

The model of games with a priori unions has proven useful in modeling the bargaining power of groups of similar agents in a wide range of cost-sharing problems. Examples include the design of airport landing fees for different types of aircraft grouped into airlines \citep{VazquezBrage1997}, the sharing of connection costs when agents are grouped into streets or cities \citep{Ber2010}, or the agreement among owners of apartments in a building to install an elevator and share corresponding costs by grouping together homeowners on the same floor \citep{Alo2020}.

Furthermore, from the dawn of the study of cooperative cost allocation games, as shown in \cite{Driessen1988}, two solution concepts with favorable properties, straightforward expressions for concrete problems, and widely accepted interpretations have been explored: the Shapley and Tijs values. While the nucleolus is another solution with desirable properties in cost allocation problems, its computational complexity has led to its exclusion from this work.
Notable applications of the Shapley and Tijs values include the Tennessee Valley Authority allocation problem \citep{Driessen1988}, or the airport \citep{Driessen1988} and highway \citep{Mosquera2007} problems. 
Additionally, in the case of the Tijs value in the highway game, it is worth mentioning that it yields cost allocations closely aligned with the fundamental concept of proportional allocation, as previously discussed in the work of \cite{Villarreal1985}. The first extensions of these two solutions to the context of games with a priori unions are the Owen value \citep{Owen1977} and the coalitional Tijs value \citep{CasasMendez2003}. 
It is well known that these two values allow for a two-stage interpretation of cost allocation. Initially, an allocation is made between the different a priori unions, and subsequently within the players constituting each union. While the Owen value employs the Shapley value in both stages, the coalitional Tijs value uses the Tijs value. 
In our setup, we present formulations for the Owen value and the coalitional Tijs value for the case of generalized highway games with a priori unions, along with axiomatic characterizations of both values.

It should be taken into account that the choice of principles guiding the fair distribution of benefits or costs is a permanent topic of debate in economic analysis. In \cite{Choudhury2021}, a proposed solution takes the form of an average between the Shapley value and an egalitarian solution. This approach accounts for the sizes of the different coalitions, asserting that in smaller groups, principles of equity or solidarity tend to govern, whereas in larger groups, marginalist criteria, inherent in the Shapley value, tend to prevail.
Similarly, \cite{Kamijo2009} indicates that a coalition often exhibits a predisposition towards a generous reallocation of surplus among its members. This inclination is illustrated through different examples from the realms of human evolution, psychology, or scientific labor management. 
\cite{Calvo2013} argue that in the presence of groups, negotiations between a priori unions and among players within the same a priori union need not adhere to the same rules. It is posited that in the latter case, players are inclined towards more solidarity-oriented solutions. The authors introduce a coalitional value operating at two levels, denoted as the Shapley-solidarity value. Payments to different unions are determined using the Shapley value, while players within each union receive compensation based on the solidarity value \citep{Nowak1994}. 
One of the properties of the solidarity value is that it is more benevolent towards null players, allocating them a portion of the generated profit. While this feature may be advantageous in benefit-sharing problems \citep{Calvo2013}, a notable challenge emerges in the specific context of the cost-sharing problem of a highway: if a player does not use any section of the highway, they would still be charged a small fee. 
A similar difficulty is encountered with the value proposed by \cite{Kamijo2009}, wherein, in a game with a priori unions, a null player may receive a non-null assignment if the union to which they belong is not a null player in the game played by the unions. 
In contrast, such a situation does not arise when employing the Tijs value, according to which players that do not use any section will be assigned a payment of $0$. 
Furthermore, unlike what happens with the Shapley value, the allocation proposed by the Tijs value facilitates the financing of less-used sections by those with higher utilization, demanding a heightened  level of solidarity among the agents. 
These considerations, coupled with the observation that the alliance between unions may be detrimental to them according to the coalitional Tijs value, but not with the Owen value, motivate the definition of a new coalitional value at two levels. We therefore introduce the Shapley-Tijs value, which utilizes the Shapley value to distribute costs among different unions and the Tijs value for allocation within each union. For this newly proposed value, a straightforward expression is assumed and is also axiomatically characterized.

Finally, the three values explored throughout this paper, along with the methodology for distinguishing between different types of vehicles and user groups, are illustrated using real data from the AP-9 highway in Spain. Notably, the fares obtained for grouped users generally result in a reduction compared to fares calculated without utilizing the a priori unions--a strategy commonly employed in real-life negotiations aimed at obtaining bonuses for certain groups of users.

The remainder of this paper is organized as follows. Section~\ref{sec:prel} introduces the cooperative game theory concepts necessary to address the problem at hand. Section~\ref{sec:main} focuses on tailoring the formulations related to the Owen value and coalitional Tijs value of the airport game to the context of the generalized highway game. We provide characterizations of these values and examine for each case how the assigned cost is affected by the formation of groups. Additionally, a new coalitional value, the Shapley-Tijs value, is introduced and axiomatically characterized. 
In Section~\ref{sec:ap9}, we utilize traffic data from the Spanish AP-9 highway to allocate fixed costs using our model and compare the different solution approaches analyzed in Section~\ref{sec:main}. Finally, Section~\ref{sec:conclusions} summarizes the main conclusions of our study.

\section{Preliminaries}\label{sec:prel}

In this section, we present some basic notions and results concerning cost games and highway problems as well as their extensions to the existence of externalities modeled by means of a priori unions, which will be useful in the rest of the paper.

\subsection{Cost games and generalized highway problems}

A \textit{cost transferable utility (TU) game} is defined by a pair $(N,c)$, where $N$ is the finite set of players and $c\colon 2^{N} \rightarrow \mathbb{R}$ is the cost function, which satisfies $c(\emptyset)=0$. We usually interpret $c(S)$ as the maximum cost that coalition $S\subseteq N$ must assume by itself. In order to simplify notation, we will use $c(i)$ instead of $c(\{i\})$ for all $i\in N$. Also, $G(N)$ represents the set of all cost TU games with set of players $N$. We say that a \textit{value} is a map $f \colon G(N)\rightarrow\mathbb{R}^{\abs{N}}$ that assigns to each game $(N,c)\in G(N)$ a vector $f(N,c)=(f_{i}(N,c))_{i\in N}$, which has the information of the cost allocated to each player $i\in N$. 
Prominent values are the so-called Shapley and Tijs values, introduced by \cite{Shapley1953} and \cite{Tijs1981}, respectively.
Let $(N,c)\in G(N)$ be a cost TU game. The Shapley value \citep{Shapley1953} is defined by the vector $(\Phi_i(N,c))_{i\in N}$ such that for all $i\in N$,
\begin{equation*}
    \Phi_i(N,c)=\sum_{S\subseteq N\backslash \{i\}}\frac{\abs{S}!\cdot(\abs{N}-\abs{S}-1)!
}{\abs{N}!}\cdot\left(c(S\cup \{i\})-c(S)\right).
\end{equation*}
A game $(N,c)\in G(N)$ is said to be \textit{concave} if for all $i\in N$ and all $S, T \subseteq  N\setminus\{i\}$, with $S \subseteq T$, $c(T\cup\{i\})-c(T)\leq c(S\cup\{i\})-c(S)$. The game is \textit{convex} when the opposite inequality is satisfied.
If $(N,c)$ is concave, the Tijs value \citep{Tijs1981} is defined by the vector $(\tau_i(N,c))_{i\in N}$ such that for all $i\in N$,
\begin{equation*}
    \tau_i(N,c)=m_i(N,c)+\alpha\cdot\left(M_i(N,c)-m_i(N,c)\right),
\end{equation*}
where $\alpha \in [0,1]$ is such that $\sum_{i\in N}\tau_i(N,c)=c(N)$, and 
$M_i(N,c)=c(N)-c(N\backslash \{i\})$ and 
$m_i(N,c)=\min_{\{S\subseteq N\mid i\in S\}}\{c(S)-\sum_{j\in S\backslash\{i\}}M_j(N,c)\}$ 
are the utopia payoff and the lower payoff of $(N,c)$ for player $i$, respectively.
Let $(N,c)\in G(N)$, it is said that $i\in N$ is a \textit{null player} in $(N,c)$ if for each $S\subseteq N$, $c(S\cup \{i\})=c(S)$. Two players $i, j\in N$ are called \textit{symmetric} in $(N,c)$ if, for each $S\subseteq N\setminus \{i, j\}$, it holds that $c(S\cup \{i\})=c(S\cup \{j\})$. Let $(N,c_{1}), (N,c_{2})\in G(N)$, the sum game $(N,c_{1}+c_{2})\in G(N)$ is defined by $(c_{1}+c_{2})(S)=c_{1}(S)+c_{2}(S)$ for all $S\subseteq N$. The game $(N,c)$ is said to be monotone if $c(S) \leq c(T)$ for each $S \subseteq T \subseteq N$.

Now we recall the concept of generalized highway problem, which was introduced in \cite{Sudholter2017}. To do so, we will use the notation presented in \cite{Kuipers2013} but employing $K$ instead of $M$ to denote sections, in order to avoid confusion with the set of a priori unions, which appears later in this paper.

\begin{definition}[\citeauthor{Sudholter2017}, \citeyear{Sudholter2017}]\label{def:gen_high_prob}
A \textit{generalized highway problem} is a 4-tuple $\Gamma=(N, K, C, T)$, where $N$ is a finite set of agents, $K$ is a finite nonempty ordered set of sections, $C \colon K\rightarrow \mathbb{R}_{+}$ with $C(t)$ representing the cost of section $t\in K$, and $T\colon N\rightarrow 2^{K}\backslash \{\emptyset\}$ with $T(i) \subseteq K$ representing the set of sections used by agent $i \in N$. Additionally, each section is required to be used by at least one agent, that is, $\cup_{i\in N}T(i)=K$. 
\end{definition}

We denote the set of generalized highway problems by $\mathcal{H}^{*}$. Given a generalized highway problem, $\Gamma=(N,K,C,T)\in \mathcal{H}^{*}$, and following \cite{Kuipers2013}, its associated cost TU game is defined by $c(S) = C(T(S))$ for all $S\subseteq N$, with $T(S)=\cup_{i\in S} T(i)$ and $C(K')=\sum_{t\in K'}C(t)$ for all $K' \subseteq K$. A cost TU game is said to be a \textit{generalized highway game} if it is the associated game of a generalized highway problem. 

In \cite{Mosquera2007}, the expressions of the Shapley and Tijs values are obtained for the non-generalized highway problem, which considers that players use only connected sections. However, since this condition is not used in the proofs, the following results are still valid for the generalized case.

\begin{proposition}[\citeauthor{Mosquera2007}, \citeyear{Mosquera2007}]\label{prop:2.2}
Let $\Gamma=(N,K,C,T)\in \mathcal{H}^{*}$ be a generalized highway problem, $(N, c)$ its associated game, and $\Phi$ the Shapley value. Then,
\begin{equation*}
    \Phi_{i}(N,c)=\sum_{t\in T(i)}\frac{C(t)}{\abs{N_{t}}},
\end{equation*}
for all $i\in N$, with $N_{t}=\{j\in N \mid t\in T(j)\}$ the set of agents that use section $t$, for each $t\in K$.
\end{proposition}

For the expression of the Tijs value, we will adopt a notation different from that of \cite{Mosquera2007}, which will be particularly useful when proving the results proposed in this paper. 

\begin{definition}\label{def:partition}
Let $\Gamma = (N,K,C,T)\in \mathcal{H}^{*}$ be a generalized highway problem, we define the \textit{set of exclusive use sections} by $K^{e}=\{t\in K \mid \abs{N_{t}}=1\}$ and the \textit{set of shared use sections} by $K^{s}=\{t\in K \mid \abs{N_{t}}>1\}$. Notice that $\{K^{e}, K^{s}\}$ is a partition of $K$, i.e.,  $K = K^{e}\cup K^{s}$ and $K^{e}\cap K^{s}=\emptyset$. 
Also, if $K^{e}\neq \emptyset$, let $T^{e}\colon N\rightarrow 2^{K^{e}}$ be defined by $T^{e}(i)=T(i)\cap K^{e}$ for all $i\in N$, with $T^{e}(i)$ representing the \textit{set of exclusive use sections of player $i\in N$}. We will also define the \textit{set of shared use sections of player $i\in N$} by $T^{s}(i)=T(i)\setminus T^{e}(i)$. 

\end{definition}

\begin{definition}
Let $\Gamma = (N,K,C,T)\in \mathcal{H}^{*}$ be a generalized highway problem, we say that the generalized highway problem $\Gamma|_{K'}=(N|_{K'}, K', C|_{K'}, T^{K'})$ is the \textit{restriction of  $\Gamma$ to the set of sections $K'\subseteq K$, $K'\neq \emptyset$}, where:
\begin{enumerate}
    \item $N|_{K'}=\{i\in N\, |\, T(i)\cap K'\neq \emptyset \}$.
    \item $C|_{K'}\colon K'\rightarrow \mathbb{R}_{+}$ is the restriction of $C$ to $K'$.
    \item $T^{K'}\colon N|_{K'}\rightarrow 2^{K'}\backslash \{ \emptyset\}$ is the function defined by $T^{K'}(i)=T(i)\cap K'$. 
\end{enumerate}

\end{definition}

\begin{remark}\label{obs:Restr}
The restriction of $\Gamma|_{K'}$ to $K''\subseteq K'$ coincides with $\Gamma|_{K''}$. That is, $\restr{\Gamma|_{K'}}{K''}=\Gamma|_{K''}$.
\end{remark}

In analogy with $\Gamma$, we can consider the associated cost TU games for the generalized highway problems $\Gamma|_{K^{e}}$ (if $K^{e}\neq \emptyset$) and $\Gamma|_{K^{s}}$ (if $K^{s}\neq \emptyset$), $(N|_{K^{e}}, c^{e})$ and $(N|_{K^{s}},  c^{s})$, respectively. In particular, let $c^{e}(S)=C(T^{e}(S))$ for all coalition $S\subseteq N|_{K^{e}}$ and $c^{s}(S)=C(T^{s}(S))$ for all coalition $S\subseteq N|_{K^{s}}$, where $T^{e}(S)= \cup_{i\in S} T^{e}(i)$ and $T^{s}(S)=\cup_{i\in S} T^{s}(i)$ for all coalition $S\subseteq N$. 

\begin{proposition}[\citeauthor{Mosquera2007}, \citeyear{Mosquera2007}] \label{prop:Tau-v_ap}
Let $\Gamma=(N,K,C,T)\in \mathcal{H}^{*}$ be a generalized highway problem, $(N, c)$ its associated game, and $\tau$ the Tijs value. Then, 
\begin{equation*}
\tau_{i}(N, c)=\begin{cases}
c^{e}(i) & \textrm{if } \; K^{s}=\emptyset \\
c^{e}(i)+c^{s}(N)\cdot \displaystyle{\frac{c^{s}(i)}{\sum_{j \in N}c^{s}(j)}} & \textrm{if } \; K^{s}\neq\emptyset, \\
\end{cases}
\end{equation*}
for all $i\in N$.
\end{proposition}

\subsection{Cost games and generalized highway problems with a priori unions (or a coalitional structure)}\label{sec:2.2}

A \textit{cost game with a coalitional structure} is a triple $(N, c, P)$, where $(N, c)$ is a cost TU game and $P = \{P_{1}, \dots, P_{A}\}$ is a partition of the player set $N$, where each set $P_{a}, \; a\in \{1,\dots,A\}$, contains the players of a specific a priori union. We denote by $U(N)$ the set of all cost games with a coalitional structure and set of players $N$. We usually identify each set $P_{a}$ with its index $a$ and denote the set of a priori unions by $M = \{1, \dots, A\}$. 

To allocate the total cost among the a priori unions, an approach similar to that of \cite{Owen1977} is adopted. We use the so-called \textit{quotient game} $(M, c_{P})\in G(M)$, defined by the set of a priori unions, $M$, and the cost function $c_{P}(H)=c(\cup_{a\in H}P_{a})$ for all coalition $H\subseteq M$. We define a \textit{value for games with a coalitional structure} as a map $f\colon U(N)\rightarrow\mathbb{R}^{|N|}$.

\cite{Owen1977} introduced the Owen value for games with a priori unions, which consists in applying the Shapley value twice: first on the quotient game and then among the players of each union. Let $(N,c,P)\in U(N)$ be a cost game with a coalitional structure. The Owen value \citep{Owen1977} is defined by the vector $(\Psi_i(N,c,P)_{i\in N})$ such that for all $i\in$ $P_a\in P$,
\begin{equation*}
    \Psi_i(N,c,P)=\sum_{H\subseteq M\backslash \{a\}} \sum_{S\subseteq P_a\backslash \{i\}}\frac{1}{\abs{P_a}\cdot A} \cdot \frac{1}{\binom{A-1}{\abs{H}}\cdot\binom{\abs{P_a}-1}{\abs{S}}} \cdot \left(c(R\cup S\cup \{i\})-c(R\cup S)\right),
\end{equation*}
where $R=\cup_{b\in H}P_b$. Furthermore, \citeauthor{Owen1977} characterized this value through its properties, as presented below.

\begin{definition}
Let $f \colon U(N)\rightarrow\mathbb{R}^{\abs{N}}$ be a value in $U(N)$. We define the following properties.
\begin{itemize}
    \item \textbf{Efficiency}: For all $(N,c,P)\in U(N)$, $\sum_{i\in N}f_{i}(N,c,P)=c(N)$. 
    \item \textbf{Null player property}: For all $(N,c,P)\in U(N)$, if player $i\in N$ is a null player in $(N,c)$, then $f_{i}(N,c,P)=0$.
    \item \textbf{Symmetry within unions}: For all $(N,c,P)\in U(N)$, if players $i,j \in N$ are symmetric in $(N,c)$ and $i, j \in P_{a}$ with $P_{a} \in P$, then $f_{i}(N,c,P)=f_{j}(N,c,P)$.
    \item \textbf{Symmetry between unions}: For all $(N,c,P)\in U(N)$, if two unions $a, b \in M$ are symmetric in the quotient game $(M,c_{P})$, then $\sum_{i\in P_{a}}f_{i}(N,c,P)=\sum_{i\in P_{b}}f_{i}(N,c,P)$. 
    \item \textbf{Additivity}: For all $(N,c_{1},P)$, $(N,c_{2},P) \in U(N)$, $f(N,c_{1}+c_{2},P)=f(N,c_{1},P)+f(N,c_{2},P)$.
\end{itemize}
\end{definition}

\begin{theorem}[\citeauthor{Owen1977}, \citeyear{Owen1977}] \label{teo:owen2} 
The Owen value, $\Psi$, is the unique value in $U(N)$ that satisfies efficiency, null player property, symmetry within unions, symmetry between unions, and additivity. 
\end{theorem}

\cite{CasasMendez2003} introduced the coalitional Tijs value for games with a priori unions, which consists in applying the Tijs value twice: first on the quotient game and then among the players of each union. Let $(N,c,P) \in U(N)$ be a cost game with a coalitional structure, where $(N,c)$ is a concave cost game. The coalitional Tijs value is defined by the vector $(\mathcal{T}_i(N,c,P))_{i\in N}$ such that for all $i\in P_a\in P$,
\begin{equation}\label{Eq:coal_tijs}
    \mathcal{T}_i(N,c,P)=m_i(N,c,P)+\alpha_a\cdot \left(M_i(N,c,P)-m_i(N,c,P)\right),
\end{equation}
where $\alpha_a\in [0,1]$ is such that $\sum_{i\in P_a}\mathcal{T}_i(N,c,P)=\tau_a(M,c_P)$, and $M_i(N,c,P)=c(N)-c(N\backslash \{i\})$ and $m_i(N,c,P)=\min_{\{S\in P(a)\mid i\in S\}}\{c(S)-\sum_{j\in S\backslash\{i\}}M_j(N,c,P)\}$ are the utopia payoff and the lower payoff of $(N,c,P)$ for player $i$, respectively, with $P(a) = \{S\subseteq N \mid S=\cup_{l\in L}{P_{l}} \cup T \; \textrm{for some} \; L\subseteq M\setminus\{a\} \; \textrm{and} \; T\subseteq P_{a}\}$.
 
In the analysis of highway problems, externalities can arise when, for example, we discriminate between different types of vehicles or certain characteristics of the trips they make (a vehicle can make a specific round trip every day). This fact motivates the following definition.

\begin{definition}

Let $\Gamma=(N, K,C,T)\in \mathcal{H}^{*}$ be a generalized highway problem and $(N,c)$ its associated game. Let also $P=\{P_{1}, \dots, P_{A}\}$ be a partition of $N$. The pair $(\Gamma, P)$ is called \textit{a generalized highway problem with a coalitional structure} and $(N, c, P)$ its \textit{associated game with a coalitional structure}. We denote the set of generalized highway problems with a coalitional structure by $\mathcal{HE}^{*}$. Furthermore, given $(\Gamma, P)\in \mathcal{HE}^{*}$, the \textit{restriction of $P$ to the set of sections $K'\subseteq K, K'\neq \emptyset$}, is defined by $P\restrict{K'} = \{P_{a}\cap N\restrict{K'} \mid P_{a}\in P,\ T(P_{a})\cap K'\neq \emptyset\}$.

\end{definition}

\section{Main results} \label{sec:main}

In this section, we present the main results of our paper. We provide straightforward formulas for the Owen value and the coalitional Tijs value in generalized highway games with a priori unions.\footnote{ 
Note that highway games are concave and monotone \citep{Kuipers2013} and that result can be immediately extended to the generalized case since it does not involve the condition of connected sections.}
We also investigate certain properties of these two values concerning their behavior when a group of unions merges into one larger union, similar to the findings of \cite{VazquezBrage1997} in the case of the Owen value for airport games. Additionally, we propose axiomatic characterizations of these values in terms of the generalized highway problem, similar to how \cite{Sudholter2017} characterized the core, the Shapley value, and the prenucleolus for the highway problem. We also introduce the Shapley-Tijs value for generalized highway games with a priori unions, which consists in computing the Shapley value among unions and using the Tijs value to allocate the corresponding costs within the unions.

\subsection{Owen value}

Let $\Gamma=(N, K, C, T) \in \mathcal{H}^{*}$ be a generalized highway problem and $P=\{P_{1},\dots,P_{A}\}$ a partition of $N$. Following the notation of \cite{VazquezBrage1997}, we denote the set of a priori unions that use section $t\in K$ by $\mathscr{A}_{t} = \{a\in M \mid t\in T(P_{a})\}$ and the set of agents of an a priori union $a\in M$ that use section $t$ by $N_{t}^{a} = \{i\in P_{a} \mid t\in T(i)\}$. 
The expression of the Owen value in the case of generalized highway games is given below. 

\begin{proposition}\label{th:owen}
Let $(\Gamma, P)\in \mathcal{HE}^{*}$, with $\Gamma=(N, K, C, T)$, be a generalized highway problem with a coalitional structure and $(N, c, P)$ its associated game with a coalitional structure. The Owen value of the game $(N,c,P)$ is given by
\begin{equation*}
\Psi_{i}(N, c, P)=\sum_{t\in T(i)}\frac{C(t)}{\abs{\mathscr{A}_{t}}\cdot \abs{N_{t}^{a}}}, 
\end{equation*}
for all $i\in P_{a}\in P$.
\end{proposition}

\begin{proof}  
For all $t\in K$, we define the cost game with a priori unions $(N, c^{t}, P)$ by
\begin{equation*}
    c^{t}(S)=\begin{cases} 
      C(t) & \textrm{if} \ t\in T(S) \\
      0 & \textrm{otherwise}, \\
  \end{cases}
\end{equation*}
for all $S\subseteq N$. 
Note that, in $(N, c^{t}, P)$, all players that do not use section $t$ are null players. By the property of additivity of the Owen value, it suffices to obtain $\Psi(N,c^{t},P)$ and then apply that $c=\sum_{t\in K}{c^{t}}$.
Consider the quotient game $(M,c^t_P)$, with $M=\{1, \dots, A\}$ and $c^t_P(H)=c^t(\cup_{a\in H}P_a)$ for all 
$H\subseteq M$. Observe also that
\begin{equation*}
    c^{t}_P(H)=\begin{cases} 
      C(t) & \textrm{if} \ t\in T(\cup_{a\in H}P_a) \\
      0 & \textrm{otherwise}, \\
  \end{cases}
\end{equation*}
for all $H\subseteq M$. Note that, in $(M,c^t_P)$, all unions of players that use section $t$ are symmetric. Therefore, for $a\in M$, using the efficiency, the null player property, and the symmetry between unions of the Owen value,
\begin{equation*}
\sum_{i\in P_{a}}\Psi_{i}(N,c^{t},P)= 
    \begin{cases} 
      \frac{C(t)}{\abs{\mathscr{A}_{t}}} & \textrm{if}\  t \in T(P_{a})	\\
      0 &  \textrm{otherwise}. \\
  \end{cases}    
\end{equation*} 
To complete the proof, let us notice now that, in $(N,c^t,P)$, all agents that use section $t$ are symmetric. 
By the properties of symmetry within unions and null player of the Owen value, if $i\in P_{a}\in P$,
\begin{equation*}
\begin{split}
    \Psi_{i}(N, c^{t},  P) =
    \begin{cases} 
          \frac{\sum_{j \in P_{a}}\Psi_{j}(N, c^{t}, P)}{\abs{N_{t}^{a}}} & \textrm{if} \ t\in T(i) \\
          0 & \textrm{otherwise} \\
      \end{cases} = 
      \begin{cases} 
          \frac{C(t)}{\abs{\mathscr{A}_{t}}\cdot \abs{N_{t}^{a}}} & \textrm{if} \ t\in T(i) \\
          0 & \textrm{otherwise}. \\
      \end{cases}    
\end{split}
\end{equation*}
Finally, using the additivity of the Owen value, we get 
\begin{equation*}
    \Psi_{i}(N, c, P)=\sum_{t\in T(i)}\frac{C(t)}{\abs{\mathscr{A}_{t}}\cdot \abs{N_{t}^{a}}},
\end{equation*}
for all $i\in P_{a}\in P$.
\end{proof} 

\cite{VazquezBrage1997} study, for the airport problem with a priori unions, what happens when there is an alliance between unions, and shows that, if costs are distributed following the Owen value, the alliance between unions is always beneficial for them. We will analyze the situation in the case of generalized highway problems, for which we need to introduce some previous notation. 

  Let $(\Gamma, P) \in \mathcal{HE}^{*}$, with $\Gamma=(N, K, C, T)$, be a generalized highway problem with a coalitional structure, $(N, c, P)$ its associated game with a coalitional structure, and $M=\{1, \dots, A\}$  the set of a priori unions. The total payment assigned to $a\in M$ by the Owen value will be denoted by $\Psi_{a}(N, c,P)=\sum_{i\in P_{a}}\Psi_{i}(N, c, P)$. We say that unions $\{1,\dots,a\}\subseteq M$ with $a\leq A$ are allied (or form an alliance) if they merge into a single union, $a^{*}$. This results in a new generalized highway game with a coalitional structure, $(N, c, P^{*})$, with the same sections and the same costs to be distributed, but with a different partition of a priori unions, defined by $ P^{*}=\{P_{a^{*}},P_{a+1}, \dots, P_{A}\}$, where $P_{a^{*}}=\cup_{\alpha=1}^{a}P_{\alpha}$ and $M^{*}=\{a^{*}, a+1, \dots, A\}$.

\begin{proposition}\label{prop:owprofit}
Let $(\Gamma, P)\in \mathcal{HE}^{*}$, with $\Gamma=(N, K, C, T)$, be a generalized highway problem with a coalitional structure, $(N, c, P)$ its associated game with a coalitional structure, and $M=\{1,\dots,A\}$ the set of a priori unions. Let $a^{*}$ be the alliance that resulted from merging unions $\{1,\dots,a\}\subseteq M$ (with $2\leq a \leq A$). Then, 
\begin{equation*}
    \Psi_{a^{*}}(N, c, P^{*})\leq \sum_{\alpha \in \{1, \dots, a\}}\Psi_{\alpha}(N, c, P),
\end{equation*}
where the inequality is strict if and only if there exists at least one section used by at least two unions from the alliance and by, at least, another union that is not part of the alliance.
\end{proposition}

\begin{proof}
For the original game with a priori unions, $(N,c,P)$, we employ the usual notation $\mathscr{A}_{t}$, while for the game resulted from merging unions $\{1, \dots, a\}$, $(N, c, P^{*})$, we use $\mathscr{A}_{t}^{*} = \{b\in M^{*} \mid t\in T(P_{b}^{*})\}$.
We have that
\begin{equation*}
    \Psi_{a^{*}}(N,c,P^{*})= \sum_{t \in T(P_{a^{*}})}\frac{C(t)}{\abs{\mathscr{A}^{*}_{t}}}
\end{equation*}
and
\begin{equation*}
    \begin{split}
    \sum_{\alpha \in \{1, \dots, a\}}\Psi_{\alpha}(N, c, P) = 
    \sum_{\alpha \in \{1, \dots, a\}}\sum_{t\in T(P_{\alpha})} \frac{C(t)}{\abs{\mathscr{A}_{t}}} =
    \sum_{t\in T(P_{a^{*}})}\sum_{\substack{\alpha \in \{1, \dots, a\}:\\ t\in T(P_{\alpha})}} \frac{C(t)}{\abs{\mathscr{A}_{t}}}  
    =\sum_{t\in T(P_{a^{*}})}\frac{\mathfrak{a}_{t}\cdot C(t)}{\abs{\mathscr{A}_{t}}}, 
    \end{split}
\end{equation*}
with $\mathfrak{a}_{t}=\abs{\{\alpha\in \{1, \dots, a\} \mid t\in T(P_{\alpha})\}}$.
Therefore, it suffices to check that, for all $t\in T(P_{a^{*}})$, it holds that
\begin{equation}\label{eq:assumption1}
    \frac{C(t)}{\abs{\mathscr{A}_{t}^{*}}}\leq \frac{\mathfrak{a}_{t}\cdot C(t)}{\abs{\mathscr{A}_{t}}}.
\end{equation}
Notice that $\abs{\mathscr{A}_{t}^{*}}=\abs{\mathscr{A}_{t}} - \mathfrak{a}_{t} + 1$. Moreover, $\mathfrak{a}_{t}\geq 1$ for all $t\in T(P_{a^{*}})$, i.e., if $t$ is a section used by the alliance $a^{*}$, then there exists at least one union in such alliance that used $t$. Returning to \eqref{eq:assumption1}, we have that 
\begin{equation*}
    \begin{split}
        & \qquad \quad \frac{C(t)}{\abs{\mathscr{A}_{t}}-\mathfrak{a}_{t}+1}\leq \frac{\mathfrak{a}_{t}\cdot C(t) }{\abs{\mathscr{A}_{t}}} \iff \frac{1}{\abs{\mathscr{A}_{t}}-\mathfrak{a}_{t}+1}\leq \frac{\mathfrak{a}_{t}}{\abs{\mathscr{A}_{t}}}  \\& \iff \mathfrak{a}_{t}(\abs{\mathscr{A}_{t}}-\mathfrak{a}_{t}+1)-\abs{\mathscr{A}_{t}}\geq 0 \iff (\abs{\mathscr{A}_{t}}-\mathfrak{a}_{t})(\mathfrak{a}_{t}-1)\geq 0,
    \end{split}
\end{equation*}
and, due to $\abs{\mathscr{A}_{t}}\geq \mathfrak{a}_{t}$ and $\mathfrak{a}_{t} \geq 1$, the inequality is proved.

Now, if $t\in T(P_{a^{*}})$, \eqref{eq:assumption1} is strict if $\abs{\mathscr{A}_{t}}>\mathfrak{a}_{t}$ and $\mathfrak{a}_{t}>1$. The first condition establishes that section $t$ is used by at least one union outside the alliance and the second condition guarantees that section $t$ is used by two or more unions from the alliance. This proves the result. 
\end{proof}

The previous proposition includes a novelty with respect to the result presented in \cite{VazquezBrage1997}. Those authors only required the condition $a<A$ for the strict inequality to hold. This is because all players (and thus all unions) use the first section in the airport game. By requesting the alliance to consist of more than one a priori union, it is already certain that there will be a section $t\in K$, the first one, in which $\mathfrak{a}_{t}>1$.
In the case of the highway, this has an interesting interpretation, which is that two agents only benefit from negotiating together if they share the use of some section.

\subsection{Coalitional Tijs value}

The coalitional Tijs value, $\mathcal{T}$, was first introduced in \cite{CasasMendez2003}. In that work, authors obtained its  expression for the airport game with a priori unions. This subsection will extend that result to the case of the generalized highway game.

To compute the coalitional Tijs value of the generalized highway game with a coalitional structure, we first compute the Tijs value of each alliance in the quotient game of $(N, c, P)$. The amount allocated to each alliance $a\in M$ will be denoted by $\mathcal{T}_{a}(N, c, P)=\sum_{i\in P_{a}}\mathcal{T}_{i} (N, c, P)=\tau_{a}(M,c_{P})$, where the second equality is straightforwardly obtained by the definition of $\mathcal{T}$. 

We shall notice that, when considering the quotient game, we have to be careful with the sections that are shared or exclusive because a section could be exclusive to one union but shared by many agents of that union. Therefore, we introduce the following notation.

\begin{definition}\label{def:sections_coal}
Let $(\Gamma, P)\in \mathcal{HE}^{*}$, with $\Gamma=(N, K, C, T)$, be a generalized highway problem with a coalitional structure and $(N, c, P)$ its associated game with a coalitional structure. We say that a section $t \in K$ is an \textit{exclusive use section in the quotient game} if $\abs{\mathscr{A}_{t}}=1$. We denote the set of exclusive use sections in the quotient game by $K_{P}^{e}$. Hence, we have a partition in $K$, $K=K_{P}^{e}\cup K_{P}^{s}$ and $K_{P}^{e}\cap K_{P}^{s}=\emptyset$, where $K^{s}_{P}$ is the set of \textit{shared use sections in the quotient game}. 

We define the function $T_{P}^{e} \colon M\rightarrow 2^{K_{P}^{e}}$ by $T_{P}^{e}(a) = T(P_{a})\cap K_{P}^{e}$ for all union $a \in M$, which represents the set of exclusive use sections used by that union. Also, we define $T^{e}_{P}(H) = \bigcup_{a\in H}T_{P}^{e}(a)$, for each $H\subseteq M$. In addition, for all $a\in M$ and all $H\subseteq M$, $c_{P}^{e}(a) = C(T_{P}^{e}(a))$ and $c_{P}^{e}(H) = C(T_{P}^{e}(H))$. Analogously, we can define $T_{P}^{s}$ and $T_{P}^{s}(H)$ and $c_{P}^{s}(H)$, for each $H\subseteq M$.

\end{definition}

In order to obtain the formula of the coalitional Tijs value for the generalized highway game with a coalitional structure, we will first show how the allocation between unions is performed. 

\begin{proposition}\label{prop:Tau-v_C_Pa}
Let $(\Gamma, P)\in \mathcal{HE}^{*}$, with $\Gamma=(N, K, C, T)$, be a generalized highway problem with a coalitional structure and $(N, c, P)$ its associated game with a coalitional structure. The Tijs value of the quotient game $(M,c_P)$ is given by

\begin{equation*}
\tau_a(M,c_P)=\mathcal{T}_{a}(N, c, P) =
\begin{cases}
    c_{P}^{e}(a) & \textrm{if } K_{P}^{s}=\emptyset \\
    c_{P}^{e}(a)+c_{P}^{s}(M) \cdot \displaystyle{\frac{c_{P}^{s}(a)}{\sum_{b \in M}c_{P}^{s}(b)}} & \textrm{if } K_{P}^{s}\neq \emptyset, \\
\end{cases}    
\end{equation*}
for all $a\in M$.
\end{proposition}

\begin{proof}
It is enough to apply Proposition~\ref{prop:Tau-v_ap} to the quotient game $(M, c_{P})$, making use of Definition~\ref{def:sections_coal}.
\end{proof}

To calculate the coalitional Tijs value, $\mathcal{T}(N, c, P)$, $\mathcal{T}_{a}(N, c, P)$ is distributed among the members of $P_{a}$, for each union $a\in M$.

\begin{proposition}\label{teo:tau_coal}
Let $(\Gamma, P)\in \mathcal{HE}^{*}$, with $\Gamma=(N, K, C, T)$, be a generalized highway problem with a coalitional structure and $(N, c, P)$ its associated game with a coalitional structure. The coalitional Tijs value of the game $(N,c,P)$ is given by
\begin{equation*}
\mathcal{T}_{i}(N, c, P) = 
\begin{cases}
c^{e}(i) & \textrm{if } K^s=\emptyset \\
c^{e}(i)+(\mathcal{T}_{a}(N, c, P)-c^{e}(a)) \cdot \displaystyle{\frac{c^{s}(i)}{\sum_{j \in P_{a}}c^{s}(j)}} & \textrm{if } K^{s}\neq\emptyset, \\
\end{cases}
\end{equation*}
with $c^{e}(a)=\sum_{j\in P_{a}}c^{e}(j)$, for all $i\in P_{a}\in P$. 
\end{proposition}

\begin{proof}
Let $i\in P_a\in P$.
First, if $K^s=\emptyset$, then $M_i(N,c,P)=c^e(i)$ and $K^s_P=\emptyset$. By Proposition~\ref{prop:Tau-v_C_Pa},  $\mathcal{T}_{a}(N, c, P)=c^e_P(a)$. Also, note that in this case $c^{e}_P(a)=\sum_{j\in P_{a}}c^{e}(j)$. Now, from \cite{Kuipers2013}, we know that the game $(N,c)$ is monotone, from which 
$m_i(N,c,P)=c^e(i)-\sum_{j\in N\backslash P_a}c^e(j)$. Using the definition of the coalitional Tijs value presented in \eqref{Eq:coal_tijs}, and having derived the utopia and lower payoffs of $(N,c,P)$ for player $i$, it directly follows that $\mathcal{T}_{i}(N, c, P)=c^e(i).$ 

To analyze the case where $K^s\not=\emptyset$, we will use arguments similar to that found in \cite{Mosquera2007} when proving Proposition \ref{prop:2.2}. Let us initially assume that $\abs{N_t}>1$ for all $t\in K$. It is then easily obtained that $M_i(N,c,P)=0$. Furthermore, due to $(N,c)$ being monotone, $m_i(N,c,P)=c^s(i)$. Applying again the definition of the coalitional Tijs value, the calculation leads to 
\begin{equation*}
    \mathcal{T}_{i}(N, c, P)=\mathcal{T}_{a}(N, c, P)\cdot \displaystyle{\frac{c^{s}(i)}{\sum_{j \in P_{a}}c^{s}(j)}}.
\end{equation*}
To complete the proof, it suffices to note that the coalitional Tijs value is covariant under strategic equivalence, as stated in \cite{CasasMendez2003}. 
As it is well known, this means that given two cost games with a priori unions, $(N,c,P)$ and $(N,c',P)$, $d>0$, and $(a_i)_{i\in N}\in \Bbb{R}^N$ such that $c(S)=d\cdot c'(S)+\sum_{i\in S}a_i$ for each $S\subseteq N$ (we say that $(N,c,P)$ and $(N,c',P)$ are strategically equivalent), then we have 
$\mathcal{T}_{i}(N, c, P)=d\cdot \mathcal{T}_{i}(N, c', P)+a_i$ for each $i\in N$. 
If $\abs{N_t}=1$ for some $t\in K$, the games $(N,c,P)$ and 
$(N,c\restrict{K\backslash \{t\}},P)$, resulting from $(N,c,P)$ when section $t$ is excluded, are strategically equivalent, as highlighted by \cite{Mosquera2007}.
Leveraging the property of covariance under strategic equivalence of the coalitional Tijs value, we can then conclude the proof.
\end{proof}

As it was done for the Owen value, it is worth asking if the alliance between unions is always beneficial for them when the costs are allocated using the coalitional Tijs value.

Before stating the results, it is necessary to point out that the partitions used in Definitions \ref{def:partition} and \ref{def:sections_coal} depend on the problem $\Gamma\in\mathcal{H}^{*}$ and $(\Gamma,P)\in \mathcal{HE}^{*}$ considered, respectively. When alliances are formed between the a priori unions, it is possible that there are sections in the quotient game that go from being shared use sections in the original game, $(N,c,P)$, to being exclusive use sections in the game with the alliance, $(N,c,P^{*})$. This motivates the following definition.

\begin{definition}\label{def:partition2}
Let $(\Gamma,P)\in \mathcal{HE}^{*}$, with $\Gamma=(N, K, C, T)$, be a generalized highway problem with a coalitional structure, $(N,c,P)$ its associated game with a coalitional structure, $M=\{1, \dots, A\}$ the set of a priori unions, and $(N, c, P^{*})$ the game resulting from the alliance, $a^{*}$, among the unions $\{1, \dots, a\}\subseteq M$. We say that $(N, c, P)$ is the \textit{original game} and $(N, c, P^{*})$ is the \textit{modified game}. In the same way, $(M, c_{P})$ and $(M^*,c_{P^{*}})$ will be the \textit{original} and \textit{modified quotient games}, respectively.
\end{definition}

Consider a partition of $K$ consisting of the exclusive use sections in both quotient games, $K_{P^*}^{ee}$; the shared use sections in the original quotient game and exclusive use in the modified quotient game, $K_{P^*}^{se}$; and the shared use sections in both quotient games, $K_{P^*}^{ss}$. Note that indeed $K=K_{P^*}^{e e}\cup K_{P^*}^{s e}\cup K_{P^*}^{s s}$ and that $K_{P^*}^{e e}$, $K_{P^*}^{s e}$, and $K_{P^*}^{s s}$ are mutually disjoint. The set of exclusive use sections in both quotient games used by the union $a\in M$, $T^{ee}_{P^{*}}(a)\subseteq K^{ee}_{P^{*}}$, is defined by $T^{ee}_{P^{*}}(a)=\bigcup_{i\in P_{a}}T(i)\cap K^{ee}_{P^*}$. The sets $T^{se}_{P^*}(a)$ and $T^{ss}_{P^*}(a)$ are defined in a similar way.

We will see below that if an alliance $a^{*}$ is formed such that $K^{se}_{P^{*}}=\emptyset$, i.e., no shared use section in the original quotient game becomes an exclusive use section in the modified quotient game, such an alliance will be beneficial.

\begin{proposition}\label{teo:TauC-B1}
Let $(\Gamma,P)\in \mathcal{HE}^{*}$, with $\Gamma=(N,K,C,T)$, be a generalized highway problem with a coalitional structure, $(N, c, P)$ its associated game with a coalitional structure, and $M=\{1,\dots, A\}$ the set of a priori unions. Let $a^{*}$ be the alliance resulting from merging the unions $\{1,\dots,a\}\subseteq M$ (with $2\leq a\leq A$). Let also $(N, c, P^{*})$ be the modified game with a coalitional structure. If $K_{P^*}^{s e}=\emptyset$, then $$\mathcal{T}_{a^{*}}(N,c,P^{*})\leq \sum_{\alpha \in \{1, \dots,a\}}\mathcal{T}_{\alpha}(N,c,P),$$ 
where the inequality is strict if and only if the alliance shares the use of some highway section and there is a section used by at least two unions of the alliance.
\end{proposition}

\begin{proof}
Two cases will be distinguished. First, we will assume that $K^{ss}_{P^{*}}=\emptyset$, which implies that $c^s_{P^*}(a^*)=0$. Because $K^{se}_{P^*}=\emptyset$, then
$c^s_{P}(\alpha)=0$ for all $\alpha \in \{1, \dots, a\}$. Therefore,
\begin{equation*}
    \begin{split}
    \mathcal{T}_{a^{*}}(N,c,P^{*})= c^{e}_{P^*}(a^{*})= \sum_{\alpha \in \{1, \dots, a\}}c^{e}_{P}(\alpha)
    =\sum_{\alpha \in \{1, \dots, a\}}\mathcal{T}_{\alpha}(N, c, P),
    \end{split}
\end{equation*}
where the second equality holds because the  sections used by the allied unions are of exclusive use in the original quotient game. Thus, in this case, we have the result. 

Second, we will assume that $K^{ss}_{P^{*}}\neq \emptyset$. If $c^s_{P^*}(a^*)=0$, we can reason as in the previous case and the result is true. Let us then assume that $c^s_{P^*}(a^*) > 0$. 
Then, 
\begin{equation*}\label{eqn:Eq6}
\begin{split}
    \mathcal{T}_{a^{*}}(N, c, P^{*}) &= c^{e}_{P^*}(a^{*}) + c^{s}_{P^*}(M^{*})\cdot \frac{c^s_{P^*}(a^{*})}{\sum_{b\in M^{*}}c^s_{P^*}(b)}. 
\end{split}
\end{equation*}
In addition,  
\begin{equation*}\label{eqn:Eq7}
\begin{aligned}
    \sum_{\alpha \in \{1, \dots, a\}}\mathcal{T}_{\alpha}(N, c, P) &    =\sum_{\alpha \in \{1, \dots, a\}}
    c^e_P(\alpha)
    +c^s_P(M)\cdot\frac{\sum_{\alpha \in \{1, \dots, a\}}
    c^s_P(\alpha)}{\sum_{b\in M}c^s_P(b)}.
   \end{aligned}
\end{equation*}
Note that $c^{e}_{P^*}(a^{*})=\sum_{\alpha \in \{1, \dots, a\}}
    c^e_P(\alpha)$ and $c^s_{P^*}(M^*)=c^s_{P}(M)$, so it is enough to prove that
\begin{equation*}
    \frac{c^s_{P^*}(a^{*})}{\sum_{b\in M^{*}}c^s_{P^*}(b)}\leq
    \frac{\sum_{\alpha \in \{1, \dots, a\}}
    c^s_P(\alpha)}{\sum_{b\in M}c^s_P(b)}
\end{equation*}
or, equivalently,
\begin{equation}
\label{eq:ineq}
    \frac{c^s_{P^*}(a^{*})}
    {\sum_{b\in M^{*}\backslash \{a^*\}}c^s_{P^*}(b)+c^s_{P^*}(a^*)}
    \leq
    \frac
{\sum_{\alpha \in \{1, \dots, a\}}c^s_P(\alpha)}
    {
    \sum_{b\in M\backslash \{1,\dots,a\}} c^s_P(b)
    +\sum_{\alpha \in \{1, \dots, a\}}c^s_P(\alpha)}.
\end{equation}

Now, note that $\sum_{b\in M^{*}\backslash \{a^*\}}c^s_{P^*}(b)=\sum_{b\in M\backslash \{1,\dots,a\}} c^s_P(b)$, which is strictly positive because $c^{s}_{P^{*}}(a^{*}) > 0$. Finally, inequality \eqref{eq:ineq} follows from the fact that the function $f(x)=\frac{x}{y+x}$, with $x\geq 0$, $y>0$, $x\neq -y$ is monotonic non-decreasing and that $\sum_{\alpha \in \{1, \dots, a\}}c^s_P(\alpha) \geq c^s_{P^*}(a^{*})$. Clearly, the inequality is strict if and only if  $\sum_{\alpha \in \{1, \dots, a\}}c^s_P(\alpha) > c^s_{P^*}(a^{*})$ or, equivalently, there exist two unions
$\alpha, \beta \in \{1, \dots, a\}$ such that $T(P_{\alpha})\cap T(P_{\beta})\neq \emptyset$. This completes the proof.
\end{proof}

The following example illustrates that the alliance is not necessarily beneficial if there is some section that changes from a shared use section in the original quotient game to an exclusive use section in the modified quotient game. 

\begin{Example}

Let $((N, K, C, T), P)\in \mathcal{HE}^{*}$, with $M=\{1,\dots,104\}$, $K=\{t_{1},t_{2}\}$, $C(t_{1})=C(t_{2})=1$, and
\begin{equation*}
T(a)=\begin{cases}
\{t_{1}\} & \textrm{if } a\in\{1, 2\}\\
\{t_{2}\} & \textrm{if } a\in M\backslash\{1, 2\}.\\
\end{cases}
\end{equation*}

An alliance $a^{*}$ is formed by merging unions $\{1,2,3,4\}$. Let $(N, c, P)$ and $(N, c, P^{*})$ be the original and modified games, respectively (see Definition~\ref{def:partition2}). Then, 
\begin{equation*}
    \mathcal{T}_{a^{*}}(N, c, P^{*}) = 1 + 1 \cdot \frac{1}{101}=\frac{102}{101}>1
\end{equation*}
and
\begin{equation*}
    \sum_{\alpha=1}^{4}\mathcal{T}_{\alpha}(N, c, P)=4\cdot \frac{2}{104}=\frac{1}{13}<1.
\end{equation*}
It can be seen that the unions are disadvantaged by having formed the alliance. Note that section $t_{1}$ has gone from being a shared use section in the original quotient game (used by two unions that are now part of the alliance) to being an exclusive use section in the modified quotient game (used by $a^{*}$).\sexample
\end{Example}

\subsection{Shapley-Tijs value}

The definition of the Shapley-Tijs value for generalized highway games with a priori unions is presented below.

\begin{definition}\label{def:S-T}
Let $(\Gamma,P)\in \mathcal{HE}^*$, with $\Gamma=(N, K, C, T)$, be a generalized highway problem with a coalitional structure and $(N, c, P)$ its associated game with a coalitional structure. We define the Shapley-Tijs value by the vector $(\Lambda_{i}(N, c, P))_{i\in N}$ such that for all $a\in M$ and all $i\in P_{a}\in P$, 

\begin{equation*}
\Lambda_{i}(N, c, P) =
\begin{cases}
c^{e}(i) & \textrm{if } K^s=\emptyset \\
c^{e}(i) + (\Psi_{a}(N,c,P)-c^{e}(a)) \cdot \displaystyle{\frac{c^{s}(i)}{\sum_{j \in P_{a}}c^{s}(j)}} & \textrm{if } K^s\neq\emptyset. \\
\end{cases}    
\end{equation*}
\end{definition}

\begin{remark}
It can be seen that the previous expression differs from the one given for the coalitional Tijs value in a generalized highway problem with a coalitional structure because we have 
\begin{equation*}
    \sum_{i\in P_{a}}\Lambda_{i}(N, c, P)=\Psi_{a}(N, c, P)=\Phi_{a}(M,c_{P})
\end{equation*}
for all $P_{a}\in P$, while
\begin{equation*}    
  \sum_{i\in P_{a}}\mathcal{T}_{i}(N, c, P)=\mathcal{T}_{a}(N, c, P)=\tau_{a}(M,c_{P}). 
\end{equation*}
\end{remark}


\subsection{Characterization of coalitional values in generalized highway problems}

We start by defining the various properties that will be used to characterize these values.

\begin{definition}
Let $\sigma\colon\mathcal{HE}^{*}\rightarrow \mathbb{R}^{\mid N\mid}$ be a value on $\mathcal{HE}^{*}$. We define the following properties. 

\begin{itemize}

    \item \textbf{Pareto optimality} (PO): For all $(\Gamma, P)\in \mathcal{HE}^{*}$, $\sum_{i\in N}\sigma_{i}(\Gamma, P)=C(K)$.
    
    \item \textbf{Equal treatment property for agents} (ETPA): For all $(\Gamma, P)\in \mathcal{HE}^{*}$, $P_{a}\in P$, and $i,j\in P_{a}$, $T(i)=T(j)$ implies $\sigma_{i}(\Gamma, P)=\sigma_{j}(\Gamma, P)$.
    \item \textbf{Equal treatment property for unions} (ETPU): For all $(\Gamma, P)$ $\in \mathcal{HE}^{*}$ and $P_{a},P_{a'}\in P$, $T(P_{a})=T(P_{a'})$ implies $\sum_{i\in P_{a}}\sigma_{i}(\Gamma, P)=\sum_{i\in P_{a'}}\sigma_{i}(\Gamma, P)$.
    
    \item \textbf{Individual independence of outside changes} (IIOC): For all $(\Gamma, P)$, $(\Gamma', P')$ $\in \mathcal{HE}^{*}$, with $\Gamma=(N, K, C, T)$, $\Gamma'=(N', K', C', T') $, and $i\in N\cap N'$, $\Gamma\restrict{T(i)}=\Gamma'\restrict{T'(i)}$ and
    $P\restrict{T(i)}=P'\restrict{T'(i)}$
    implies $\sigma_{i}(\Gamma, P)=\sigma_{i}(\Gamma', P')$.
    
    \item 
    \textbf{Coalitional independence of outside changes} (CIOC): For all $(\Gamma, P)$, $(\Gamma', P') \in \mathcal{HE}^{*}$, with $\Gamma=(N, K, C, T)$, $\Gamma'=(N', K', C', T') $, and $a\in M\cap M'$, $\Gamma\restrict{T(P_{a})}=\Gamma'\restrict{T'(P'_{a})}$ and 
    $P\restrict{T(P_a)}=P'\restrict{T'(P'_a)}$ implies $\sum_{i\in P_{a}}\sigma_{i}(\Gamma, P)=\sum_{i\in P'_{a}}\sigma_{i}(\Gamma', P')$. 
    
    \item \textbf{Proportionality in shared sections among agents} (PSSA): For all $(\Gamma,P)\in \mathcal{HE}^{*}$ and $a\in M$ with $T^{e}(P_{a})=\emptyset$, $\textrm{there exists}\ \mathfrak{c}_{a}\in \mathbb{R}_{+}$ such that for all $i\in P_{a}$, $\sigma_{i}(\Gamma, P)=\mathfrak{c}_{a}\cdot c^{s}(i)$.
    
    \item \textbf{Proportionality in shared sections among unions} (PSSU): For all $(\Gamma,P)\in \mathcal{HE}^{*}$ with $T^{e}_{P}(M)=\emptyset$, $\textrm{there exists}\ \mathfrak{c}\in \mathbb{R}_{+}$ such that for all $a\in M$, $\sum_{i\in P_{a}}\sigma_{i}(\Gamma,P)=\mathfrak{c}\cdot c^{s}_{P}(a)$.
    
    \item \textbf{Covariance under a prolongation for exclusive use by an agent} (CPEA): For all $(\Gamma, P)$, $(\Gamma', P')$ $\in \mathcal{HE}^{*}$, with
    $\Gamma=(N, K, C, T)$ and $\Gamma'=(N', K', C', T') $, if $i\in N'$, $K'=K\cup\{t\}$, $\Gamma=\Gamma'\restrict{K}$, $P=P'\restrict{K}$, and $\{j\in N' \mid t\in T'(j)\}=\{i\}$, then for all $j\in N'$ 
    \begin{equation*}
    \sigma_{j}(\Gamma', P') = 
    \begin{cases}
    \sigma_{j}(\Gamma, P) + C'(t) & \textrm{if}\ i=j\\
    \sigma_{j}(\Gamma, P) & \textrm{otherwise}.
    \end{cases}
        \end{equation*}
    
    \item \textbf{Covariance under a prolongation for exclusive use by a union} (CPEU): For all $(\Gamma,P)$, $(\Gamma',P')\in \mathcal{HE}^{*}$, with $\Gamma=(N, K, C, T)$ and $\Gamma'=(N', K', C', T') $, if $P_{a}\in P'$, $K'=K\cup\{t\}$, $\Gamma=\Gamma'\restrict{K}$,
    $P=P'\restrict{K}$, and $\{P_{b}\in P' \mid t\in T'(P_{b})\} = \{P_{a}\}$, then for all $P_{b}\in P'$
    \begin{equation*}
    \sum_{i\in P_{b}}\sigma_{i}(\Gamma', P') = 
    \begin{cases}
    \sum_{i\in P_{b}}\sigma_{i}(\Gamma, P) + C'(t) & \textrm{if}\ P_{a}=P_{b}\\
    \sum_{i\in P_{b}}\sigma_{i}(\Gamma, P) & \textrm{otherwise}.
    \end{cases}
    \end{equation*}
    
\end{itemize}
\end{definition}

(PO) means that the fees collected cover the cost of the installation. 
This property has already been used by \cite{Sudholter2017}. The following properties are introduced for the first time in this paper, although they are adaptations of axioms already existing in the literature for group-free highway problems. Here, we extend those axioms to the coalitional context, both to the level of a priori unions and to the agent level within a single union.  
Although \cite{Owen1977} has previously utilized similar two-level properties to axiomatically characterize the Owen value,  their application has persisted over time, evident in works such as 
\cite{CasasMendez2003}, \cite{Lorenzo2019}, \cite{Alo2023}, and \cite{Casajus2023}.
(ETPA) and (ETPU) imply that vehicles (or coalitions of vehicles, respectively) using the same sections have to pay the same fare. 
Originating from the equal treatment property defined in \cite{Sudholter2017}, these two properties have been newly formulated. 
(IIOC) and (CIOC) imply that what a vehicle (or union, respectively) pays does not depend on how unused sections are traveled. 
These two properties have been defined from the individual independence of outside changes proposed in \cite{Sudholter2017}.
(PSSA) states that the vehicles belonging to unions that do not use exclusive use sections will make a payment proportional to the cost of the sections they use. Similarly, (PSSU) states that in a problem where no union uses exclusive use sections, the total payment of each union will be proportional to the cost of the sections used. 
Analogous properties to (PSSA) and (PSSU) are employed to characterize the Tijs value in \cite{Tijs1987} or the coalitional Tijs value in \cite{CasasMendez2003}. Lastly, 
(CPEA) and (CPEU) state that adding a section of exclusive use by one vehicle (or by a union of vehicles, respectively) does not affect the payment of the remaining vehicles (or unions). These properties have been formulated from the property of covariance under exclusive prolongation proposed in \cite{Sudholter2017}.

Given that every generalized highway problem with a coalitional structure, $(\Gamma, P)$, has an associated game, $(N, c, P )$, the value of a highway problem with a coalitional structure can be defined as a value of its associated game. The following theorem provides an axiomatic characterization of the Owen value for generalized highway problems with a coalitional structure.

\begin{theorem}\label{teo:owen}
The Owen value on $\mathcal{HE}^{*}$ is the unique solution that satisfies (PO), (ETPA), (ETPU), and (IIOC).
\end{theorem}

\begin{proof}
It is immediate to see from its expression that the Owen value of generalized highway problems with a coalitional structure, $\Psi$, satisfies (PO), (ETPA), (ETPU), and (IIOC). 

Now, let $(\Gamma, P) \in \mathcal{HE}^{*}$ and let $\sigma$ be a solution on $\mathcal{HE}^{*}$ satisfying (PO), (ETPA), (ETPU), and (IIOC). We will prove that $\sigma(\Gamma, P) = \Psi(N,c,P)$ by induction on $\abs{K}$, where $(N,c,P)$ is the game with a coalitional structure associated to $(\Gamma, P)$.

If $\abs{K}=1$, then $T(i)=T(j)\ \textrm{for all}\ i, j\in N$ and $T(P_{a})=T(P_{b})\ \textrm{for all}\ P_{a}$, $P_{b}\in P$. Thus, 
$\sigma_{i}(\Gamma, P)= \frac{C(K)}{\abs{M} \cdot \abs{P_{a}}}=\Psi_{i}(N,c,P)$, for all $a\in M$ and all $i \in P_{a}$, due to (PO), (ETPA), and (ETPU), and we have the result. 
Suppose that $\sigma(\Gamma, P)=\Psi(N,c,P)$ holds for all $2\leq l < \abs{K}$, and take $l=\abs{K}$. We define 
\begin{equation*}
    R=\{i\in N \mid T(i)\subsetneq K\}\;\; \textrm{and} \;\; Q=\{i\in N \mid T(i)=K\}.
\end{equation*}
For each $i\in R$, consider the restriction $(\Gamma\restrict{T(i)}, P\restrict{T(i)})$. Since $\abs{T(i)}<\abs{K}$, by the induction hypothesis we have that $\sigma$ and $\Psi$ coincide on $(\Gamma\restrict{T(i)},  P\restrict{T(i)})$ for player $i$.
Now, given that $\restr{\Gamma\restrict{T(i)}}{T(i)}=\Gamma\restrict{T(i)}$, $\restr{P\restrict{T(i)}}{T(i)}=P\restrict{T(i)}$,  and that $\sigma$ and $\Psi$ satisfy (IIOC), 
\begin{equation*}
    \sigma_{i}(\Gamma, P)=\sigma_{i}(\Gamma\restrict{T(i)}, P\restrict{T (i)})= \Psi_{i}(N,c, P).
\end{equation*}
In this way, $
\sigma_{i}(\Gamma, P)=\Psi_{i}(N,c, P)\ \textrm{for all}\ i\in R.    
$
If $R=N$, we have finished. If $R\subsetneq N$, then $Q\neq \emptyset$. Let us see now that $\sigma_{i}(\Gamma, P)=\Psi_{i}(N,c, P)\ \textrm{for all}\ i\in Q$, if $Q\neq \emptyset$.

Let $H=\{a\in M \mid T(P_{a})=K\}$. 
Note that, for all $i\in Q$, there exists $a \in H$ such that $i\in P_{a}$. Because $\sigma$ and $\Psi$ satisfy (PO) and (ETPU) and given that  $\sigma_{i}(\Gamma,P)=\Psi_{i}(N,c, P)$, for all $i \in R$, we obtain 
$\sigma_{a}(\Gamma,P)=\Psi_{a}(N,c, P)$ for all $a\in H$. To conclude, two cases are considered. 
If $a\in H$ and $P_a\cap R= \emptyset$, by (ETPA) we obtain $\sigma_{i}(\Gamma,P)=\Psi_{i}(N,c, P)$, for all $i \in P_{a}.$ Finally, if $a\in H$ and $P_a\cap R\neq \emptyset$, again by (ETPA) and using that $\sigma_{i}(\Gamma,P)=\Psi_{i}(N,c, P)$ for all $i \in R$, we have that $\sigma_{i}(\Gamma,P)=\Psi_{i}(N,c, P)$, for all $i \in P_{a}\cap Q.$
\end{proof}

Next, the logical independence of the properties used in Theorem~\ref{teo:owen} is shown.
\begin{itemize}
    \item If we define for all $(\Gamma, P)\in \mathcal{HE}^{*}$,  $f_{i}(\Gamma,P)=0$ for all $i\in N$, it satisfies all the properties except for (PO).

    \item Let $f$ be the solution defined, for all $(\Gamma, P)\in \mathcal{HE}^{*}$, by 
    \begin{equation*}
    f_{i}(\Gamma,P) = 
        \begin{cases}
        \displaystyle\sum_{t\in T(i)}{\dfrac{C(t)}{\abs{\mathscr{A}_{t}}}} & \textrm{if}\ i=\min\{j\in N_{t}^{a}\} \\
        0 & \textrm{otherwise},
        \end{cases}
    \end{equation*}
    for all $P_{a}\in P$, for all $i\in P_{a}$. This solution does not satisfy (ETPA) but it satisfies the rest of the properties.

    \item To prove the independence of (ETPU) we define for all $(\Gamma, P)\in \mathcal{HE}^{*}$ and for all $i\in N$, $f_{i}(\Gamma,P)=\Phi_{i}(N,c)$, where $\Phi$ denotes the Shapley value.  

    \item To prove the independence of (IIOC) we define for all $(\Gamma, P)\in \mathcal{HE}^{*}$ and for all $i\in N$, $f_{i}(\Gamma,P)=\mathcal{T}_{i}(N,c,P)$, where $\mathcal{T}$ denotes the coalitional Tijs value. 
    
\end{itemize}

 We will now give a characterization of the coalitional Tijs value for generalized highway problems with a coalitional structure, just as it has been done for the Owen value.

\begin{theorem}\label{teo:coal_tijs}
The coalitional Tijs value on $\mathcal{HE}^{*}$ is the unique solution that satisfies (PO), (PSSA), (PSSU), (CPEA), and (CPEU).
\end{theorem}

\begin{proof}
From Proposition~\ref{teo:tau_coal}, it follows that the coalitional Tijs value satisfies (PO), (PSSA), (PSSU), (CPEA), and (CPEU). 

Now, let $(\Gamma, P)\in \mathcal{HE}^{*}$ and let $\sigma$ be a solution on $\mathcal{HE}^{*}$ satisfying the five properties of the statement of the theorem. We will prove that $\sigma(\Gamma, P)=\mathcal{T}(N,c,P)$. 
By (PSSU) and (CPEU), there exists $\mathfrak{c}\in \mathbb{R}_{+}$ such that 
\begin{equation*}
   \sum_{i\in P_{a}}\sigma_{i}(\Gamma,P)=c_{P}^{e}(a)+ \mathfrak{c}\cdot c_{P}^{s}(a), 
\end{equation*}
for all $a\in M$. 
By (PO), it follows that 
\begin{equation*}
    \mathfrak{c}=\frac{c_P(M)-\sum_{b\in M}c_{P}^{e}(b)}{\sum_{b\in M}c^{s}_{P}(b)}=\frac{c_{P}^{s}(M)}{\sum_{b\in M}c_{P}^{s}(b)}
\end{equation*}
if $K^s_P\neq \emptyset$, and $\mathfrak{c}=0$ if $K^s_P=\emptyset.$ Therefore, it holds that 
\begin{equation*}
    \sum_{i\in P_{a}}\sigma_{i}(\Gamma, P)=\mathcal{T}_{a}(N, c, P), 
\end{equation*}
for all $a\in M$. 
This, together with (PSSA) and (CPEA), and following a reasoning similar to that of the first part of the proof, implies that 
\begin{equation*}
    \sigma_{i}(\Gamma, P)=\mathcal{T}_{i}(N, c, P),
\end{equation*}
for all $i\in N$, thus having the result.
\end{proof}

Next, the logical independence of the properties used in Theorem~\ref{teo:coal_tijs} is shown.
\begin{itemize}

    \item (PO) is independent of the rest of the properties. For the proof we define for all $(\Gamma, P)\in \mathcal{HE}^{*}$,
    \begin{equation*}
    f_{i}(\Gamma,P) = c^{e}(i) + \left(E_{a}(N,c,P) - c^{e}(a)\right)\cdot \dfrac{c^{s}(i)}{\sum_{j\in P_{a}}{c^{s}(j)}},
    \end{equation*}
     for all $P_{a}\in P$, for all $i\in P_{a}$, where
    \begin{equation*}
         E_{a}(N,c,P) = \displaystyle{
         \sum_{
         \substack{t\in T(P_{a}):\\ \abs{\mathscr{A}_{t}}=1}
         }{C(t)}
         }.
    \end{equation*}

     \item (PSSA) is independent of the rest of the properties. For the proof we define for all $(\Gamma, P)\in \mathcal{HE}^{*}$ and for all $P_{a}\in P$, for all $i\in P_{a}$,  
         \begin{equation*}
    f_{i}(\Gamma,P) = 
        \begin{cases}
        c^{e}(i) & \textrm{if}\ K^{s}=\emptyset\\
        c^{e}(i) + \dfrac{\mathcal{T}_{a}(N,c,P)-c^{e}(a)}{\abs{\mathcal{N}_{s}^{a}}} & \textrm{if}\ K^{s}\not=\emptyset,
        \end{cases}
    \end{equation*}
    where $\mathcal{N}_{s}^{a}=\{j\in P_{a}\mid c^{s}(j)\not=0\}$. 

     \item To prove the independence of (PSSU) we define for all $(\Gamma, P)\in \mathcal{HE}^{*}$ and for all $i\in N$, $f_{i}(\Gamma,P)=\Lambda_{i}(N,c,P)$, where $\Lambda$ denotes the Shapley-Tijs value.
     
    \item To prove the independence of (CPEA) we define for all $(\Gamma, P)\in \mathcal{HE}^{*}$ and for all $P_{a}\in P$, for all $i\in P_{a}$,  
    \begin{equation*}
        f_{i}(\Gamma, P) = \mathcal{T}_{a}(N,c,P)\cdot \dfrac{c^{s}(i)}{\sum_{j\in P_{a}}{c^{s}(j)}}.
    \end{equation*}
     
    \item (CPEU) is independent of the rest of the properties. For the proof we define for all $(\Gamma, P)\in \mathcal{HE}^{*}$,
    \begin{equation*}
        f_{i}(\Gamma,P) = 
        \begin{cases}
        c^{e}(i) & \textrm{if}\ K^{s}=\emptyset\\
        c^{e}(i) + \mathcal{\overline{T}}_{a}(N,c,P) \cdot \dfrac{c^{s}(i)}{\sum_{j\in P_{a}}{c^{s}(j)}} & \textrm{if}\ K^{s}\not=\emptyset,
        \end{cases}
    \end{equation*}
    for all $P_{a}\in P$, for all $i\in P_{a}$,  where
    \begin{equation*}
        \mathcal{\overline{T}}_{a}(N,c,P) = \left(c_{P}(M)-\sum_{i\in N}{c^{e}(i)}\right) \cdot \dfrac{c^{s}_{P}(a)}{\sum_{b\in M}{c^{s}_{P}(b)}}.
    \end{equation*}

\end{itemize}

Finally, a characterization for the Shapley-Tijs value, defined above, will be provided. The following lemma will be used in the proof of such an axiomatic characterization. 

\begin{lemma}\label{lem:ShTj-C1}
Let $\sigma$ be a solution on $\mathcal{HE}^{*}$ that satisfies (PO), (ETPU), and (CIOC). Then, 
\begin{equation*}
    \sum_{i\in P_a}\sigma_{i}(\Gamma, P)=\Psi_{a}(N,c, P)=\sum_{t\in T(P_{a})}\frac{C(t)}{\abs{\{b\in M\mid t\in T(P_{b})\}}},
\end{equation*} 
for all $(\Gamma, P)\in \mathcal{HE}^{*},$ and all $a\in M$, with $(N,c,P)$ the associated game with a coalitional structure.
\end{lemma}

\begin{proof}
Let $(\Gamma, P)\in \mathcal{HE}^*$ and suppose that $\sigma$ satisfies (PO), (ETPU), and (CIOC). The proof will be done by induction on $\abs{K}$. 
If $\abs{K}=1$, then $T(P_{a})=T(P_{b})$ for all $a,b\in M$. Thus, due to (PO) and (ETPU), we have the result. 
Suppose that $\sum_{i\in P_{a}}\sigma_{i}(\Gamma, P)=\Psi_{a}(N,c, P)$ for all $a\in M$ holds for all $2\leq l < \abs{K}$. 
Now, take $l=\abs{K}$. We define 
\begin{equation*}
    R=\{a\in M \mid T(P_{a})\subsetneq K\} \;\; \textrm{and} \;\; Q=\{a\in M \mid T(P_{a})=K\}.
\end{equation*}
For each $a\in R$, we can consider the generalized highway problem with a coalitional structure restricted to the sections $T(P_{a}) \subsetneq K$, $(\Gamma\restrict{T(P_{a})},P\restrict{T(P_{a})})$.
%
Since $\abs{T(P_{a})}<\abs{K}$, by the induction hypothesis we have the result on $(\Gamma\restrict{T(P_{a})}, P\restrict{T(P_{a})})$
for the union $a$. 
Considering that $\restr{\Gamma\restrict{T(P_{a})}}{T(P_{a})}=\Gamma\restrict{T(P_{a})}$ for all $a\in R$ and that $\sigma$ satisfies (CIOC), we obtain 
\begin{equation*}
    \sum_{i\in P_{a}}\sigma_{i}(\Gamma, P)=\sum_{i\in P_{a}}\sigma_{i}(\Gamma\restrict{T(P_{a})}, P\restrict{T(P_{a})})=\sum_{i\in P_{a}}\Psi_{i}(N,c, P)=\Psi_{a}(N,c, P),
\end{equation*}    
where the second-to-last equality is fulfilled because $\Psi$ also satisfies (CIOC).

If $R=M$, we have the result. If $R\subsetneq M$, then $Q\neq \emptyset$. In such a case, using the first part of the proof, (PO), and (ETPU), it can be seen that 
\begin{equation*}
    \sum_{i\in P_{a}}\sigma_{i}(\Gamma, P)=\Psi_{a}(N,c, P),
\end{equation*}
for all $a\in Q$. This concludes the proof of the result.
\end{proof}

\begin{theorem}\label{teo:shapley_tijs}
The Shapley-Tijs value on $\mathcal{HE}^{*}$ is the unique solution that satisfies (PO), (ETPU), (CIOC), (PSSA), and (CPEA).
\end{theorem}

\begin{proof}
From Definition \ref{def:S-T}, it follows that the Shapley-Tijs value satisfies (PO), (ETPU), (CIOC), (PSSA), and (CPEA).

Now, let $(\Gamma, P)\in\mathcal{HE}^{*}$ and let $\sigma$ be a solution on $\mathcal{HE}^*$ satisfying these properties. We will prove that $\sigma(\Gamma,P)=\Lambda(N,c,P)$. By (PSSA) and (CPEA), there exists $\mathfrak{c}_{a}\in \mathbb{R}_{+}$ such that 
\begin{equation*}
    \sigma_{i}(\Gamma, P)= c^{e}(i)+\mathfrak{c}_{a}\cdot c^{s}(i),
\end{equation*}
for all $a\in M$, $i\in P_{a}\in P$, and $\mathfrak{c}_{a}=0$ if $K^{s}=\emptyset$.
By Lemma~\ref{lem:ShTj-C1}, since $\sigma$ satisfies (PO), (ETPU), and (CIOC), we have that 
\begin{equation*}
    \sum_{i\in P_{a}}\sigma_{i}(\Gamma, P)=\Psi_{a}(N,c, P),
\end{equation*}
for all $a\in M$. 
It is therefore obtained that for all $a\in M$ and for all $i\in P_{a}\in P$,
\begin{equation*}
    \sigma_{i}(\Gamma,P) = \Lambda_{i}(N,c, P) = 
    \begin{cases}
    c^{e}(i) & \textrm{if} \ K^{s}=\emptyset\\
    c^{e}(i) + \Big(\Psi_{a}(N,c, P)-c^{e}(a)\Big) \cdot  \displaystyle{\frac{c^{s}(i)}{\sum_{j \in P_{a}}c^{s}(j)}} & \textrm{if} \ K^{s}\neq\emptyset, \\
\end{cases}
\end{equation*}
and that concludes the proof.
\end{proof}

Next, the logical independence of the properties used in Theorem~\ref{teo:shapley_tijs} is shown.
\begin{itemize}

     \item If we define for all  $(\Gamma, P)\in \mathcal{HE}^{*}$, $f_{i}(\Gamma,P)=c^{e}(i)$ for all $i\in N$, it satisfies all the properties except for (PO).

    \item To prove the independence of (ETPU) we define for all $(\Gamma, P)\in \mathcal{HE}^{*}$, 
    \begin{equation*}
    f_{i}(\Gamma,P) = 
        \begin{cases}
        c^{e}(i) & \textrm{if}\ K^{s}=\emptyset\\
        c^{e}(i) + \overline{\Psi}_{a}(N,c,P)\cdot \dfrac{c^{s}(i)}{\sum_{j\in P_{a}}{c^{s}(j)}} & \textrm{if}\ K^{s}\not=\emptyset,
        \end{cases}
    \end{equation*}
    for all $P_{a}\in P$, for all $i\in P_{a}$, where
    \begin{equation*}
         \overline{\Psi}_{a}(N,c,P) = \displaystyle{
         \sum_{
         \substack{t\in T(P_{a})\cap K^{s}:\\ a=\min\{b\in M\mid t\in T(P_{b})\cap K^{s}\}}
         }{C(t)}
         }.
    \end{equation*}
  
    \item To show that (CIOC) is independent of the rest of the properties we define for all $(\Gamma, P)\in \mathcal{HE}^{*}$, $f_{i}(\Gamma,P)=\mathcal{T}_{i}(N,c,P)$ for all $i\in N$.

    \item To prove the independence of (PSSA) we define for all $(\Gamma, P)\in \mathcal{HE}^{*}$ and for all $i\in N$, $f_{i}(\Gamma,P)=\Psi_{i}(N,c,P)$, where $\Psi$ denotes the Owen value.

    \item (CPEA) is independent of the rest of the properties. For the proof we define for all $(\Gamma, P)\in \mathcal{HE}^{*}$, for all $P_{a}\in P$, and for all $i\in P_{a}$,  
    \begin{equation*}
        f_{i}(\Gamma, P) = \Psi_{a}(N,c,P)\cdot \dfrac{c^{s}(i)}{\sum_{j\in P_{a}}{c^{s}(j)}}.
    \end{equation*}

\end{itemize}

Table~\ref{tab:Characterizations} provides a concise overview of the characterizations for the Owen value, the coalitional Tijs value, and the Shapley-Tijs value. Each column lists the properties satisfied by each coalitional value, highlighting in bold those that characterize the solutions. It can be seen that the Shapley-Tijs value combines the proportional allocation among the members of each union (PSSA) and the independence of outside changes in the allocation of a priori unions (CIOC), property that allows the merging of unions to be advantaged, in exchange of losing the proportional allocation among unions (PSSU) that the coalitional Tijs value satisfies.

\begin{table}[H]\footnotesize
\centering
\scalebox{0.8}{
\begin{tabular}{ccc}
\midrule
$\Psi$ & $\mathcal{T}$ & $\Lambda$  \\ 
\midrule
(\textbf{PO}) & (\textbf{PO}) & (\textbf{PO}) \\
 (\textbf{ETPA}) & \textcolor{Gray}{(ETPA)} & \textcolor{Gray}{(ETPA)} \\
 (\textbf{ETPU}) & \textcolor{Gray}{(ETPU)} & (\textbf{ETPU}) \\ 
 (\textbf{IIOC}) &  &  \\
 \textcolor{Gray}{(CIOC)} &  & (\textbf{CIOC})  \\
 & (\textbf{PSSA}) &  (\textbf{PSSA})  \\
  & (\textbf{PSSU}) &   \\
  \textcolor{Gray}{(CPEA)} & (\textbf{CPEA}) & (\textbf{CPEA})  \\
 \textcolor{Gray}{(CPEU)} & (\textbf{CPEU}) & \textcolor{Gray}{(CPEU)}  \\
\midrule
\end{tabular}}
\caption{Overview of the properties satisfied by the Owen value, $\Psi$, the coalitional Tijs value, $\mathcal{T}$, and the Shapley-Tijs value, $\Lambda$. In bold, the properties that characterize these values.}
\label{tab:Characterizations}
\end{table}

\section{An application to the Spanish AP-9 highway} \label{sec:ap9}

\cite{Kuipers2013} showed how to allocate the fixed costs of the Spanish AP-68 highway among its users using cooperative games values, in particular the Shapley value and the nucleolus. However, the authors only considered vehicles of the same category (light vehicles) and therefore obtained a single price for all its users. What actually happens in practice is quite different, since there are several categories of vehicles (light, heavy 1, and heavy 2) with different rates.

In this setup, we extend the model considered in \cite{Kuipers2013} to include different categories of users, as well as the relationship that may exist between them. To do this, data from sections of the AP-9 highway connecting A Coruña and Vigo, two cities in the northwest of Spain, will be used. These sections are represented in Figure \ref{fig:EsqAP)}, where each node is an entrance/exit point of the highway and sections consist of the segments joining two consecutive nodes. 

\begin{figure}[H]
    \centering
    \scalebox{0.75}{
    \includegraphics{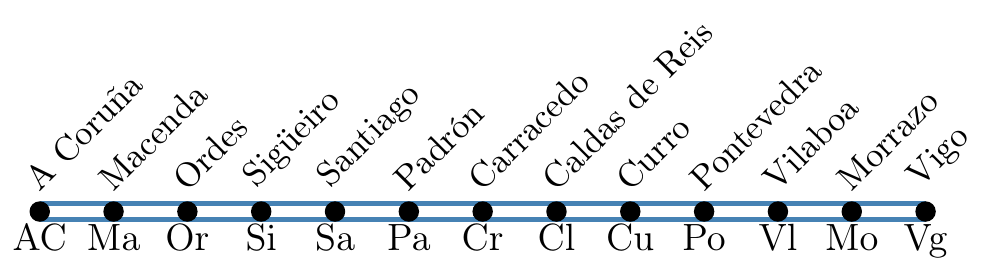}}
    \caption{Sections of the AP-9 highway between A Coruña and Vigo.}
    \label{fig:EsqAP)}
\end{figure}

The AP-9 highway users can be classified into three types, according to their type of vehicle: light, heavy 1, and heavy 2.\footnote{Light vehicles include cars, motorcycles, or vans; heavy 1 vehicles include 3-axle vehicles, such as cars with trailers, or 3-axle trucks; and heavy 2 vehicles are generally vehicles with more than 3 axles \citep{AUDASA}.} To be able to include them in the model, it is considered that each section of the highway is composed of three levels and that larger vehicles need to use more levels of the highway than smaller ones. This procedure was inspired by \cite{Fragnelli2000}. The authors distribute the costs of a railway network between slow and fast trains. Fast trains need a higher quality track and therefore a higher cost track than slow trains. This additional cost is reflected in the fact that they must use both the basic level of the track (sufficient for slow trains) and an extra level whose cost is the difference between the cost of the higher quality track and the lower quality track.

In our case, three levels will be considered, as we work with three categories. Figure~\ref{fig:my_label} shows a highway with three sections ($t_{1}, t_{2}$, and $t_{3}$) used by three vehicles, one of each type. By dividing the highway into three levels ($l\in \{0,1,2\}$), nine subsections are obtained, which will be the sections of the generalized highway problem to be considered.\footnote{For ease of visualization, we display the three subsections of each section stacked, rather than consecutively as in the classical airport problem.}
It is necessary to work with the generalized version because the subsections may not be connected. User 1 is a light vehicle that uses highway sections $t_{1}$ and $t_{2}$ and therefore uses subsections $t_{1}^{0}$ and $t_{2}^{0}$ of the problem, i.e., only the basic level. User 2 is a heavy 1 vehicle using $t_{2}$ and $t_{3}$ and therefore corresponds to subsections $t_{2}^{0}$, $t_{3}^{0}$, $t_{2}^{1}$, and $t_{3}^{1}$ of the generalized highway problem. The third user is a heavy 2 vehicle that travels $t_{1}$, $t_{2}$, and $t_{3}$, so it will use subsections $t_{1}^{l}$, $t_{2}^{l}$, and $t_{3}^{l}$, with $l\in \{0,1,2\}$.

\begin{figure}[H]
    \centering
    \scalebox{0.75}{
    \includegraphics{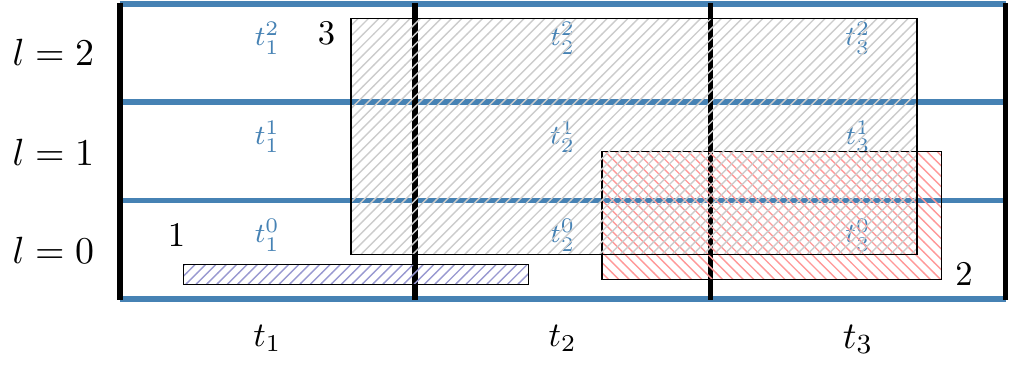}}
    \caption{Subsections used by three types of vehicles that travel through three highway sections.}
    \label{fig:my_label}
\end{figure}

In addition to distinguishing between different types of vehicles, another issue of interest in the model is the consideration of different groups of agents that can negotiate together in order to obtain lower fares.
In this application, two cases will be considered.
Firstly, the problem without a priori unions and,  secondly, the problem with an alliance among the members of the heavy 2 category.
The analysis of this grouping is motivated by a discount that this category of vehicles  recently obtained on the AP-9.
However, the theoretical amount of the discount obtained should not be compared faithfully with the actual amount of the discount, since we only contemplate the distribution of fixed costs  
and, moreover, the costs that we consider for each section are slightly modified with respect to the real ones, for ease of illustration.

For this purpose, the traffic data of the AP-9 highway from \cite{Ministerio2019a} and the prices for each journey \citep{Ministerio2019b} have been employed. The data used can be found in Table~\ref{tab:AP9Data} and are explained below. 

Firstly, the official traffic data consist in the average number of vehicles that travel through each section of the road on a daily basis, referred to as the Average Daily Index (ADI), distinguishing between light (Lg) and heavy vehicles, and corresponding to the year 2019. 
The towns considered, connected by the AP-9, are A Coruña (AC), Macenda (Ma), Ordes (Or), Sig\"ueiro (Si), Santiago (Sa), Padrón (Pa), Carracedo (Cr), Caldas de Reis (Cl), Curro (Cu), Pontevedra (Po), Vilaboa (Vl), Morrazo (Mo), and Vigo (Vg).
As the percentage of heavy 1 (H1) and heavy 2 (H2) vehicles on the AP-9 is not publicly available, it has been assumed that the ADI for heavy vehicles is evenly divided between the two heavy vehicle categories. The official fares for each section are also included, again from 2019.\footnote{It should be noted that the prices are not generally additive and only the fares for consecutive sections are included. In addition, as the Vilaboa-Morrazo and Padrón-Carracedo routes cannot be performed, the prices chosen for these sections are the difference between Pontevedra-Morrazo and Pontevedra-Vilaboa, and between Padrón-Caldas de Reis and Carracedo-Caldas de Reis, respectively.}

\begin{table}[H] \footnotesize
\centering
\begin{tabular}{crrrrrrr}
\midrule
 & \multicolumn{1}{c}{ADI Lg} & \multicolumn{1}{c}{ADI H1} & \multicolumn{1}{c}{ADI H2} & \multicolumn{1}{c}{Toll Lg} & \multicolumn{1}{c}{Toll H1} & \multicolumn{1}{c}{Toll H2} &  \\ 
\midrule
 AC-Ma & 22661 & 1032 & 1032 & 1.90 & 3.30 & 4.05 \\
 Ma-Or & 16602 & 718 & 718 & 3.05 & 5.40 & 6.60 \\ 
 Or-Si & 18082 & 808 & 808 & 1.80 & 2.95 & 3.80 \\ 
 Si-Sa & 16127 & 745 & 745 & 1.75 & 3.10 & 3.80 \\
 Sa-Pa & 22075 & 1170 & 1170 & 2.40 & 4.00 & 4.95 \\ 
 Pa-Cr & 18218 & 1027 & 1027 & 1.05 & 1.85 & 2.65 \\ 
 Cr-Cl & 20702 & 1206 & 1206 & 0.75 & 1.15 & 1.35 \\ 
 Cl-Cu & 19426 & 1095 & 1095 & 1.25 & 2.20 & 2.75 \\ 
 Cu-Po & 21948 & 1244 & 1244 & 1.35 & 2.45 & 3.05 \\ 
 Po-Vl & 29796 & 1318 & 1318 & 1.10 & 1.95 & 2.40 \\ 
 Vl-Mo & 28279 & 1229 & 1229 & 1.70 & 1.85 & 3.50 \\ 
 Mo-Vg & 61032 & 2277 & 2277 & 1.10 & 2.00 & 2.80 \\
\midrule
\end{tabular}
\caption{Traffic data (ADI) and fees of the three types of vehicle for the AP-9.}
\label{tab:AP9Data}
\end{table}

The fixed construction costs for each level in each section have been calculated in a similar manner to that of \cite{Kuipers2013}, by multiplying the price of each section by the number of users. To obtain the price of the different levels for each section, the level 0 fare is considered to be the light vehicle fare, the level 1 fare is the difference between the fare of heavy 1 vehicles and light vehicles, and the level 3 fare is the difference between the fare of heavy 2 vehicles and heavy 1 vehicles. The result of these calculations can be found in Table~\ref{tab:SectCost}. 
It should be noted that when working with daily user data and not 
annual data, the cost obtained would be interpreted as the daily fixed cost to be distributed on the highway, and the a priori unions formed only contain same-day users.

\begin{table}[H]\footnotesize
\centering
\begin{tabular}{crrrrrr}
\midrule
 &\multicolumn{1}{c}{AC-Ma} & \multicolumn{1}{c}{Ma-Or} & \multicolumn{1}{c}{Or-Si} & \multicolumn{1}{c}{Si-Sa} & \multicolumn{1}{c}{Sa-Pa} & \multicolumn{1}{c}{Pa-Cr}  \\ 
\midrule
 $l=0$ & 46977.50 & 55015.90 & 35456.40 & 30829.75 & 58596.00 & 21285.60  \\ 
 $l=1$ & 2889.60 & 3374.60 & 1858.40 & 2011.50 & 3744.00 & 1643.20  \\ 
 $l=2$ & 774.00 & 861.60 & 686.80 & 521.50 & 1111.50 & 821.60\\ 
\midrule
\end{tabular}

\begin{tabular}{crrrrrr}
& \multicolumn{1}{c}{Cr-Cl} & \multicolumn{1}{c}{Cl-Cu} & \multicolumn{1}{c}{Cu-Po} & \multicolumn{1}{c}{Po-Vl} & \multicolumn{1}{c}{Vl-Mo} & \multicolumn{1}{c}{Mo-Vg}\\
\midrule
 $l=0$ & 17335.50 & 27020.00 & 32988.60 & 35675.20 & 52252.90 & 72144.60 \\ 
 $l=1$  & 964.80 & 2080.50 & 2736.80 & 2240.60 & 368.70 & 4098.60 \\ 
 $l=2$  & 241.20 & 602.25 & 746.40 & 593.10 & 2027.85 & 1821.60 \\ 
\midrule
\end{tabular}
\caption{Fixed costs obtained for each level of each section.}
\label{tab:SectCost}
\end{table}

Once the costs of each section and the number of users have been obtained, different values are calculated for the cost games associated with the generalized highway problems without and with a priori unions. In these games, the cost associated with a coalition is determined by taking into account the sections of the highway used by at least one agent in the coalition. However, the number of agents that have used each of these sections is not taken into account, in contrast to the maintenance games defined in \cite{Fragnelli2000}.

Recall that the Shapley value of the cost game associated with a highway problem can be expressed, for each agent $i\in N$, by $\Phi_{i}(N,c)=\sum_{t\in T(i)}\frac{C(t)}{|N_{t}|}=\sum_{t\in T(i)}k_{t}\cdot C(t)$, where $k_{t}$ is a constant that only depends on section $t$, so it is sufficient to obtain the values $k_{t}\cdot C(t)$ for each $t\in K$, which play the role of a toll to be paid for using section $t$. Once the tolls $k_{t}\cdot C(t)$ for each section have been computed, each $\Phi_{i}(N,c)$, $i\in N$, can be obtained by adding the tolls for the sections used by $i$ ($t\in T(i)$). Additionally, the Tijs value can be decomposed in a similar way:
\begin{align*}
\tau_{i}(N,c)&=\sum_{t\in T^{e}(i)}C(t)+c^{s}(N)\cdot \frac{\sum_{t\in T^{s}(i)}C(t)}{\sum_{j\in N}{c^{s}(j)}}
=\sum_{t\in T(i)\cap K^{e}}C(t)+c^{s}(N) \cdot \frac{\sum_{t\in T(i)\cap K^{s}}C(t)}{\sum_{j\in N}{c^{s}(j)}}=\sum_{t\in T(i)}{k'_{t}\cdot C(t)},
\end{align*}
where
\begin{equation*}
    k'_{t}=\begin{cases} 1 & \textrm{if } t\in K^{e} \\ \displaystyle{\frac{c^{s}(N)}{\sum_{j\in N}{c^{s}(j)}}} & \textrm{if } t\in K^{s}. 
\end{cases}
\end{equation*}

The previous argument implies that obtaining the corresponding tolls for each section in the problem without a priori unions is sufficient. The results of calculating the Shapley value and the Tijs value in that case are found in Table~\ref{tab:Case1}. The coalitional values considered in this setup, for each agent, can also be decomposed into sections, but in these cases, it should be noted that the constants depend not only on the sections but also on the a priori unions that use each section or to which each agent belongs.

\begin{table}[h!] \footnotesize
    \centering  
\begin{tabular}{c@{\hspace{0.7 cm}}rrr c rrr}
\midrule
 & \multicolumn{3}{c}{$\Phi$} &&  \multicolumn{3}{c}{$\tau$} \\
 \cmidrule{2-4} \cmidrule{6-8}
 & \multicolumn{1}{c}{Lg} & \multicolumn{1}{c}{H1} & \multicolumn{1}{c}{H2} &&  \multicolumn{1}{c}{Lg} & \multicolumn{1}{c}{H1} & \multicolumn{1}{c}{H2} \\ 
\midrule
  AC-Ma & 1.90 & 3.30 & 4.05 & & 1.68 & 1.79 & 1.82 \\ 
  Ma-Or & 3.05 & 5.40 & 6.60 & & 1.97 & 2.09 & 2.12 \\ 
  Or-Si & 1.80 & 2.95 & 3.80 & & 1.27 & 1.34 & 1.36 \\ 
  Si-Sa & 1.75 & 3.10 & 3.80 & & 1.11 & 1.18 & 1.20 \\ 
  Sa-Pa & 2.40 & 4.00 & 4.95 & & 2.10 & 2.24 & 2.28 \\ 
  Pa-Cr & 1.05 & 1.85 & 2.65 & & 0.76 & 0.82 & 0.85 \\ 
  Cr-Cl & 0.75 & 1.15 & 1.35 & & 0.62 & 0.66 & 0.66 \\ 
  Cl-Cu & 1.25 & 2.20 & 2.75 & & 0.97 & 1.04 & 1.07 \\ 
  Cu-Po & 1.35 & 2.45 & 3.05 & & 1.18 & 1.28 & 1.31 \\ 
  Po-Vl & 1.10 & 1.95 & 2.40 & & 1.28 & 1.36 & 1.38 \\ 
  Vl-Mo & 1.70 & 1.85 & 3.50 & & 1.87 & 1.89 & 1.96 \\ 
  Mo-Vg & 1.10 & 2.00 & 2.80 & & 2.59 & 2.73 & 2.80 \\  
\midrule
\end{tabular}
    \caption{Shapley value and Tijs value for the generalized highway problem without a priori unions.}
    \label{tab:Case1}
\end{table}

The fares obtained with the Shapley value are identical to the original ones due to how the costs of each section have been considered, and this already occurred in \cite{Kuipers2013}. It can be seen how the Tijs value disadvantages users of highly used sections such as Morrazo-Vigo (Mo-Vg), while heavy 2 vehicles are the least affected, since level 2 sections are the least used. 
As presented in  Table~\ref{tab:Characterizations}, all of the values that we consider satisfy the property of Pareto optimality (PO), that is, the sum of the allocations to each player is equal to the total costs to be distributed. Naturally, in order to determine highway tolls, the actual values provided by the corresponding allocation must be rounded to two decimal places. This rounding is the reason why the proposed allocation of the Tijs value in Table~\ref{tab:Case1} does not satisfy (PO). A possible way to deal with this issue is to ceiling the results obtained, thus ensuring the recovery of the highway's total costs. The exact results are available on request from the authors.

The results of the alliance between the heavy 2 vehicles can be found in Table~\ref{tab:Case2}. 
In this case, there are many a priori unions because each light or heavy 1 vehicle gives rise to an individual union. Nevertheless, it is not necessary to provide the rates for each of them due to the symmetry in the ratio. In addition, the other a priori union considered contains all the heavy 2 vehicles and, therefore, the rates can be divided into only three categories. It can be observed how the alliance of the heavy 2 vehicles achieves a significant discount at the cost of slightly increasing the prices of the other two categories.

\begin{table}[h!]\footnotesize
\centering
\resizebox{9.5cm}{!}{\begin{tabular}{c@{\hspace{0.7 cm}} rrr c rrr c rrr}
      \midrule
\multicolumn{1}{c}{} & \multicolumn{3}{c}{$\Psi$} & & \multicolumn{3}{c}{$\mathcal{T}$} & & \multicolumn{3}{c}{$\Lambda$} \\ 
\cmidrule{2-4} \cmidrule{6-8} \cmidrule{10-12}
 \multicolumn{1}{c}{}& \multicolumn{1}{c}{Lg} & \multicolumn{1}{c}{H1} & \multicolumn{1}{c}{H2} & &\multicolumn{1}{c}{Lg} & \multicolumn{1}{c}{H1} & \multicolumn{1}{c}{H2} & &\multicolumn{1}{c}{Lg} & \multicolumn{1}{c}{H1} & \multicolumn{1}{c}{H2}\\
     \midrule
AC-Ma & 1.98 & 4.78 & 0.75 & & 1.73 & 1.84 & 0.85 & & 1.98 & 4.78 & 0.83 \\ 
  Ma-Or & 3.18 & 7.87 & 1.21 & & 2.02 & 2.14 & 0.99 & & 3.18 & 7.87 & 0.97 \\ 
  Or-Si & 1.88 & 4.18 & 0.85 & & 1.30 & 1.37 & 0.64 & & 1.88 & 4.18 & 0.62 \\ 
  Si-Sa & 1.83 & 4.53 & 0.70 & & 1.13 & 1.20 & 0.56 & & 1.83 & 4.53 & 0.55 \\ 
  Sa-Pa & 2.52 & 5.72 & 0.95 & & 2.15 & 2.29 & 1.07 & & 2.52 & 5.72 & 1.04 \\ 
  Pa-Cr & 1.11 & 2.71 & 0.80 & & 0.78 & 0.84 & 0.40 & & 1.11 & 2.71 & 0.39 \\ 
  Cr-Cl & 0.79 & 1.59 & 0.20 & & 0.64 & 0.68 & 0.31 & & 0.79 & 1.59 & 0.30 \\ 
  Cl-Cu & 1.32 & 3.22 & 0.55 & & 0.99 & 1.07 & 0.50 & & 1.32 & 3.22 & 0.49 \\ 
  Cu-Po & 1.42 & 3.62 & 0.60 & & 1.21 & 1.31 & 0.61 & & 1.42 & 3.62 & 0.60 \\ 
  Po-Vl & 1.15 & 2.85 & 0.45 & & 1.31 & 1.39 & 0.65 & & 1.15 & 2.85 & 0.63 \\ 
  Vl-Mo & 1.77 & 2.07 & 1.65 & & 1.92 & 1.93 & 0.92 & & 1.77 & 2.07 & 0.90 \\ 
  Mo-Vg & 1.14 & 2.94 & 0.80 & & 2.65 & 2.80 & 1.31 & & 1.14 & 2.94 & 1.28 \\
\midrule
\end{tabular}}
\caption{Owen value, coalitional Tijs value, and Shapley-Tijs value for the generalized highway problem with a priori union between the heavy 2 vehicles.}
\label{tab:Case2}
\end{table}

Although in this case the alliance of heavy 2 vehicles makes level 2 of exclusive use in the quotient game, the reduction obtained in the other levels prevents the alliance from worsening the rates obtained by the coalitional Tijs value. As previously proven, this need not be the case, and we can illustrate it by repeating the process but considering that the level 2 sections have a cost of four times higher. These results are found in Tables~\ref{tab:Case1.x4} and \ref{tab:Case2.x4}, where it can be seen how the coalitional Tijs value worsens the rates of heavy 2 vehicles. The Shapley-Tijs value maintains the property that alliances benefit while preserving the philosophy of proportional sharing (in shared sections) within the alliance.

\begin{table}[h!]\footnotesize
    \centering  
\begin{tabular}{c@{\hspace{0.7 cm}} rrr c rrr}
\midrule
 & \multicolumn{3}{c}{$\Phi$} & & \multicolumn{3}{c}{$\tau$} \\ 
\cmidrule{2-4} \cmidrule{6-8}
 & \multicolumn{1}{c}{Lg} & \multicolumn{1}{c}{H1} & \multicolumn{1}{c}{H2} & & \multicolumn{1}{c}{Lg} & \multicolumn{1}{c}{H1} & \multicolumn{1}{c}{H2} \\ 
\midrule
  AC-Ma & 1.90 & 3.30 & 6.30 & & 1.78 & 1.89 & 2.01 \\ 
  Ma-Or & 3.05 & 5.40 & 10.20 & & 2.09 & 2.22 & 2.35 \\ 
  Or-Si & 1.80 & 2.95 & 6.35 & & 1.35 & 1.42 & 1.52 \\ 
  Si-Sa & 1.75 & 3.10 & 5.90 & & 1.17 & 1.25 & 1.33 \\ 
  Sa-Pa & 2.40 & 4.00 & 7.80 & & 2.22 & 2.36 & 2.53 \\
  Pa-Cr & 1.05 & 1.85 & 5.05 & & 0.81 & 0.87 & 0.99 \\
  Cr-Cl & 0.75 & 1.15 & 1.95 & & 0.66 & 0.70 & 0.74 \\ 
  Cl-Cu & 1.25 & 2.20 & 4.40 & & 1.03 & 1.11 & 1.20 \\  
  Cu-Po & 1.35 & 2.45 & 4.85 & & 1.25 & 1.35 & 1.46 \\  
  Po-Vl & 1.10 & 1.95 & 3.75 & & 1.35 & 1.44 & 1.53 \\ 
  Vl-Mo & 1.70 & 1.85 & 8.45 & & 1.98 & 1.99 & 2.30 \\ 
  Mo-Vg & 1.10 & 2.00 & 5.20 & & 2.74 & 2.90 & 3.18 \\ 
\midrule
\end{tabular}
    \caption{Shapley value and Tijs value for the generalized highway problem without a priori unions in which level 2 sections are four times more expensive than those in Table~\ref{tab:SectCost}.}
    \label{tab:Case1.x4}
\end{table}

\begin{table}[h!]\footnotesize
\centering
\resizebox{10cm}{!}{\begin{tabular}{c rrr c rrr c rrr}
\midrule
\multicolumn{1}{c}{} & \multicolumn{3}{c}{$\Psi$} & & \multicolumn{3}{c}{$\mathcal{T}$} & &\multicolumn{3}{c}{$\Lambda$} \\ 
\cmidrule{2-4} \cmidrule{6-8} \cmidrule{10-12}
 \multicolumn{1}{c}{}& \multicolumn{1}{c}{Lg} & \multicolumn{1}{c}{H1} & \multicolumn{1}{c}{H2} &  &\multicolumn{1}{c}{Lg} & \multicolumn{1}{c}{H1} & \multicolumn{1}{c}{H2} & & \multicolumn{1}{c}{Lg} & \multicolumn{1}{c}{H1} & \multicolumn{1}{c}{H2}\\
\midrule
 AC-Ma & 1.98 & 4.78 & 3.00 & & 1.73 & 1.83 & 3.34 & & 1.98 & 4.78 & 3.32 \\ 
  Ma-Or & 3.18 & 7.87 & 4.81 & & 2.02 & 2.15 & 3.90 & & 3.18 & 7.87 & 3.87 \\ 
  Or-Si & 1.88 & 4.17 & 3.41 & & 1.30 & 1.37 & 2.52 & & 1.88 & 4.17 & 2.51 \\ 
  Si-Sa & 1.83 & 4.52 & 2.81 & & 1.13 & 1.21 & 2.20 & & 1.83 & 4.52 & 2.19 \\ 
  Sa-Pa & 2.52 & 5.72 & 3.80 & & 2.15 & 2.29 & 4.21 & & 2.52 & 5.72 & 4.18 \\ 
  Pa-Cr & 1.11 & 2.70 & 3.20 & & 0.78 & 0.84 & 1.65 & & 1.11 & 2.70 & 1.64 \\ 
  Cr-Cl & 0.79 & 1.59 & 0.80 & & 0.64 & 0.67 & 1.21 & & 0.79 & 1.59 & 1.21 \\ 
  Cl-Cu & 1.32 & 3.21 & 2.20 & & 0.99 & 1.07 & 1.99 & & 1.32 & 3.21 & 1.97 \\ 
  Cu-Po & 1.42 & 3.62 & 2.40 & & 1.21 & 1.31 & 2.44 & & 1.42 & 3.62 & 2.43 \\ 
  Po-Vl & 1.15 & 2.85 & 1.80 & & 1.31 & 1.39 & 2.54 & & 1.15 & 2.85 & 2.52 \\ 
  Vl-Mo & 1.77 & 2.07 & 6.60 & & 1.92 & 1.93 & 3.83 & & 1.77 & 2.07 & 3.80 \\ 
  Mo-Vg & 1.14 & 2.94 & 3.20 & & 2.65 & 2.80 & 5.26 & & 1.14 & 2.94 & 5.23 \\ 
\midrule
\end{tabular}}
\caption{Owen value, coalitional Tijs value, and Shapley-Tijs value for the generalized highway problem with a priori union between the heavy 2 vehicles and being level 2 sections four times more expensive than those in Table~\ref{tab:SectCost}.}
\label{tab:Case2.x4}
\end{table}

Currently, some users make round trips within a day and receive a discount on the AP-9 toll \citep{AUDASA}. A new category of vehicles satisfying this condition could also be considered in our setup. If two identical round trips were to ally in an a priori union, that coalition would be symmetrical to another union consisting of a single trip with the same characteristics. This can be interpreted as that a round trip pays only for the outbound journey (a round trip has the same cost as a one-way trip). Consequently, to include them in the model it would be sufficient to remove same-day return trips from the user matrix and add to the resulting fares that same-day return trips have a cost of 0. This is because only fixed costs are being distributed, which do not increase even if the highway is more used. 

\section{Conclusions} \label{sec:conclusions}

This paper addresses the allocation of fixed costs in highway problems with externalities. In particular, we propose a generalization of the methodology presented in \cite{Kuipers2013}, considering a model that accommodates various vehicle categories. To tackle this situation, we adopt an approach similar to \cite{Fragnelli2000}, decomposing each highway section into different quality levels, giving rise to the so-called subsections. It is assumed that larger vehicles utilize more subsections than smaller vehicles. Given that the set of used subsections may not be connected, we employ the generalized highway problem introduced in \cite{Sudholter2017}. In addition, our model incorporates a priori unions \citep{Owen1977} to reflect potential relationships between groups of agents, such as the bargaining power within an association of truck drivers. 

To investigate the cost allocation in our setup, we extend several theoretical results on coalitional values to the generalized highway problem. Specifically, we consider the Owen value \citep{Owen1977} and the coalitional Tijs value \citep{CasasMendez2003}, and introduce the Shapley-Tijs value. This latter allocation arises as a combination of the two former ones to achieve a more equitable distribution within the unions, on the one hand, and to ensure that the alliance of unions is always beneficial to them, on the other. Furthermore, we derive efficient formulations, study properties for these coalitional values, and provide an axiomatic characterization for each of them. The methodology introduced is then applied and illustrated using a real database of the Spanish AP-9 highway. 

Upon analyzing the expression of the coalitional Tijs value applied to each union, denoted as $\mathcal{T}_a$ and presented in Proposition \ref{prop:Tau-v_C_Pa}, we observe that it can be considered as a union value in the sense of \cite{Brink2014}. In their work, they investigate two union values that generalize the Shapley value and assign payoffs to unions in a game with a coalitional structure. These values differ in their axiomatization only in the collusion neutrality property used. While player collusion neutrality states that the payoff of a union does not change if two members of that union collude, union collusion neutrality states that the collusion of two unions does not change the sum of their payoffs. Both values are studied in the context of an airport problem with a priori unions. It is worth analyzing the implications of collusion properties in the realm of highway problems and the solutions proposed in the current article.

In this work, we have adopted the use of solutions defined by a two-stage approach, resulting in a symmetric treatment of each a priori union. An alternative strategy involves assigning a distinct treatment or weight to each coalition. In existing literature, various works have employed exogenously defined weight systems. However, one can also endogenously provide a natural weight for each coalition based on its cardinality, as in \cite{Gomez2010}. The incorporation of weighted values is also observed in applications of cost games, as exemplified in \cite{Gomez2013}. In light of this, it is valuable to delve into this literature to discover novel approaches for tariff design within the context of highway games with grouped players.

Regarding other future lines of research, it would be interesting to incorporate maintenance costs into our model, in addition to fixed costs. \cite{Fragnelli2000} introduced maintenance cost games and presented an expression for the Shapley value in this context. Later, \cite{Costa2015} obtained an expression for the Owen value, and no further values have been explored yet, presumably because maintenance cost games are generally non-concave.

\section*{Acknowledgments}

This work is part of the R+D+I project grant PID2021-124030NB-C32, funded by MCIN/AEI/10.13039/\linebreak
501100011033/ and by ``ERDF A way of making Europe''/EU. This research was also funded by \textit{Grupos de Referencia Competitiva} ED431C 2021/24 from the \textit{Consellería de Cultura, Educación e Universidades, Xunta de Galicia.}

\bibliography{biblio}

\end{document}